\documentclass[reqno,11pt]{amsart}
\usepackage{hyperref}
\usepackage{amsfonts,mathrsfs,bbm,latexsym,rawfonts,amsmath,amssymb,amsthm,epsfig}
\usepackage{fullpage, setspace, color}
\usepackage{indentfirst}
\newtheorem{theorem}{Theorem}[section]
\newtheorem{lemma}{Lemma}[section]
\newtheorem{corollary}[theorem]{Corollary}

\newtheorem{thm}{Theorem}[section]

\newtheorem{prop}[thm]{Proposition}
\newtheorem{defn}{Definition}[section]

\numberwithin{equation}{section}

\renewcommand{\O}{\emptyset}

\def\<{\left\langle} \def\>{\right\rangle}
\def\({\left(} \def\){\right)}

\makeatletter 
\renewcommand{\section}{\@startsection{section}{1}{0mm}
	{-\baselineskip}{0.5\baselineskip}{\bf\leftline}}
\makeatother

\title[]
{a new metric for statistical properties of long time behaviors}


\email{zhengliqi15@mails.ucas.edu.cn}
\email{zhzheng@amt.ac.cn}

\thanks{The second author is supported by the NSF of China(No. 11671382), CAS Key Project of Frontier Sciences(No. QYZDJ-SSW-JSC003), the Key Lab. of Random Complex Structures and Data Sciences CAS and National Center for Mathematics and Interdisplinary Sciences CAS}

\thanks{$^*$ Corresponding author: Zuohuan Zheng}

\begin{document}
	\maketitle
	
	\centerline{\scshape Liqi Zheng}
	\medskip
	{\footnotesize
		\centerline{Academy of Mathematics and Systems Science, Chinese Academy of Sciences}
		\centerline{Beijing 100190, China}
		\centerline{University of Chinese Academy of Sciences}
		\centerline{Beijing 100049, China}
	} 
	
	\medskip
	
	\centerline{\scshape Zuohuan Zheng$^*$}
	\medskip
	{\footnotesize
		\centerline{College of Mathematics and Statistics, Hainan Normal University}
		\centerline{Haikou, Hainan 571158, China}
		\centerline{Academy of Mathematics and Systems Science, Chinese Academy of Sciences}
		\centerline{Beijing 100190, China}
		\centerline{University of Chinese Academy of Sciences}
		\centerline{Beijing 100049, China}
	
	}
	
	\bigskip
	
	
	\begin{abstract}
		Let $(X,T)$ be a topological dynamical system with metric $d$. We define a new function $\overline{F}(x,y)=\limsup\limits_{n \to +\infty} \inf\limits_{\sigma \in S_n} \frac 1n \sum\limits_{k=1}^n d(T^k x,T^{\sigma(k)} y)$ by using permutation group $S_n$. It's shown $F(x,y)=\lim\limits_{n \to +\infty} \inf\limits_{\sigma \in S_n} \frac 1n \sum\limits_{k=1}^n d(T^k x,T^{\sigma(k)} y)$ exists when $x,y \in X$ are generic points.
		Applying this function, we prove $(X,T)$ is uniquely ergodic if and only if $\overline{F}(x,y)=0$ for any $x,y \in X$. The characterizations of ergodic measures and physical measures by $\overline{F}(x,y)$ are given. We introduce the notion of weak mean equicontinuity and prove that $(X,T)$ is weak mean equicontinuous if and only if the time averages $f^{*}(x)=\lim\limits_{n \to +\infty}\frac 1n \sum\limits_{k=1}^n f(T^k x)$ are continuous for all $f \in C(X)$.
	\end{abstract}

	\section{Introduction}
	
	Throughout this paper, a \emph{topological dynamical system} (for short t.d.s.) is a pair $(X,T)$, where $X$ is a non-empty compact metric space with a metric $d$ and $T$ is a continuous map from $X$ to itself.
	
	When studying long time behaviors, people firstly focused on equicontinuous systems, because they have simple dynamical behaviors \cite{EJK,PJ}. But only the cumulative effect of points in orbits can influence statistical properties of long time behaviors, so it is reasonable to ignore where positions are for some points in orbits when studying statistical properties of long time behaviors. For this purpose, \emph{mean-L-stable systems} were introduced \cite{S,JA,JCO}. We call a dynamical system $(X,T)$ \emph{mean-L-stable} if for any $\varepsilon >0$, there is a $\delta>0$ such that $d(x,y)<\delta$ implies $d(T^n x,T^n y)<\varepsilon$ for all $n \in \mathbb{N}$ except a set of upper density(see Section  \ref{2} for definition) less than $\varepsilon$. Recently, Li, Tu and Ye~\cite{JSX} introduced \emph{mean equicontinuous systems}. A dynamic system is called \emph{mean equicontinuous} if for any $\varepsilon>0$, there exists a $\delta>0$ such that whenever $x,y \in X$ with $d(x,y)<\delta$,
	\[
	\limsup_{n \to +\infty} \frac 1n \sum_{k=0}^{n-1}d(T^k x,T^k y) < \varepsilon .
	\]
	In their paper, they proved that a dynamic system is mean equicontinuous if and only if it is mean-L-stable. We refer to \cite{FG1,FG2,FGL,WJLX,JL} for further study on mean equicontinuity and related subjects.
	
	The highlight in this paper is to ignore the order of points in orbits, for the order makes no sense when studying statistical properties of long time behaviors. In order to state our idea precisely, we introduce some new notations. For any $x,y \in X$, we define
	\[
	\overline{F}(x,y)=\limsup_{n \to + \infty} \inf_{\sigma \in S_n} \frac 1n \sum_{k=1}^n d(T^k x,T^{\sigma(k)} y) ,
	\]
	\[
	\underline{F}(x,y)=\liminf_{n \to + \infty} \inf_{\sigma \in S_n} \frac 1n \sum_{k=1}^n d(T^k x,T^{\sigma(k)} y) ,
	\]
	\[
	N(\overline{F}) = \{(x,y) \in X \times X|\overline{F}(x,y)=0\} ,
	\]
	\[
	N(\underline{F}) = \{(x,y) \in X \times X|\underline{F}(x,y)=0\} ,
	\]
	where $S_n$ is the n-order permutation group. If $\overline{F}(x,y)=\underline{F}(x,y)$, we say $F(x,y)$ exists, and define 
	\[
	F(x,y)=\lim_{n \to + \infty} \inf_{\sigma \in S_n} \frac 1n \sum_{k=1}^n d(T^k x,T^{\sigma(k)} y) ,
	\]
	\[
	N(F) = \{(x,y) \in X \times X|F(x,y)=0\} .
	\]
	It is easy to obtain that
	\[
	N(\overline{F})=N(F) \subset N(\underline{F}) .
	\]
	
	$\overline{F}(x,y)$ and $\underline{F}(x,y)$ are functions which can measure the difference between distributions of $\text{Orb}(x)$ and $\text{Orb}(y)$, where $\text{Orb}(x) = \{x, T x, T^2 x, \cdots \}$ and
	$\text{Orb}(y)  = \{y, T y, T^2 y, \cdots \}$ are the orbits of $x$ and $y$ respectively. When $x$ and $y$ are generic points(see Section  \ref{2} for definition), we can deduce that $\overline{F}(x,y)= \underline{F}(x,y)$.
	\begin{theorem}\label{1.2}
		Let $(X,T)$ be a t.d.s. If $x,y \in X$ are generic points, then $F(x,y)$ exists.
	\end{theorem}
	
	In \cite{S}, Fomin proved that a minimal mean-L-stable system is uniquely ergodic. And then in \cite{JCO}, Oxtoby proved a more general result that each transitive mean-L-stable system is uniquely ergodic. In our paper, we shall give new characterizations of unique ergodicity by $F(x, y)$ and $\underline{F}(x,y)$.
	\begin{theorem}\label{1.3}
		Let $(X,T)$ be a t.d.s. Then the following statements are equivalent:
		\begin{itemize}
			\item [(\emph{1})]
			$(X,T)$ is uniquely ergodic;
			\item [(\emph{2})]
			$N(F)= X \times X$;
			\item [(\emph{3})]
			$N(\underline{F})= X \times X$.
		\end{itemize}
	\end{theorem}

	In the study of invariant measures, the set $N(F)$ can play an important role. We derive that a invariant measure $\mu$ is ergodic if and only if $(\mu \times \mu)\big(N(F) \big)=1$. In the last few decades, physical  measures is a hot topic of invariant measures. We find out $(X,T)$ has physical measures(see Section  \ref{2} for definition) with respect to $m \in M(X)$ if and only if $(m \times m) \big(N(F) \bigcap (Q \times Q) \big)>0$, where $M(X)$ is the set of all regular Borel probability measures of $X$ and $Q$ is the set of all generic points. With respect to $N(\underline{F})$, we obtain the same results.       
	
	In order to make clear statement of our results, we introduce the following two notions.
	\begin{defn}
		Let $(X,T)$ be a t.d.s. We say $(X,T)$ is $\overline{F}$-continuous at $x \in X$ if for any $\varepsilon >0$, there is a $\delta >0$ such that whenever $d(x,y)<\delta$, we have $\overline{F}(x,y)<\varepsilon$. Denote by $C(\overline{F})$ all $\overline{F}$-continuous points. If $C(\overline{F}) = X$, we say $(X,T)$ is $\overline{F}$-continuous. In this case, we also call $(X,T)$ weak mean equicontinuous.
	\end{defn}
	
	\begin{defn}
		Let $(X,T)$ be a t.d.s. We say $(X,T)$ is F-continuous at $x \in X$ if for any $\varepsilon >0$, there is a $\delta >0$ such that whenever $d(x,y)<\delta$, $F(x,y)$ exists and $F(x,y)<\varepsilon$. Denote by $C(F)$ all F-continuous points. If $C(F) = X$, we say $(X,T)$ is $F$-continuous. 
	\end{defn}
    Since $X$ is compact, it is easy to deduce that $(X,T)$ is $\overline{F}$-continuous if and only if for any $\varepsilon >0$, there is a $\delta > 0$ such that whenever $x,y \in X$ with $d(x,y)<\delta$, we have $\overline{F}(x,y)<\varepsilon$. Similarly, we can derive that $(X,T)$ is $F$-continuous if and only if for any $\varepsilon >0$, there is a $\delta > 0$ such that whenever $x,y \in X$ with $d(x,y)<\delta$, $F(x,y)$ exists and $F(x,y)<\varepsilon$.
    
   Obviously, mean equicontinuity implies weak mean equicontinuity. But in general, weak mean equicontinuity does not imply mean equicontinuity. The following example is a weak mean equicontinuous but not mean equicontinuous system.
    \[
    T: S^1 \longrightarrow S^1
    \]
    \begin{equation*}
    T(x)=\begin{cases}
    1-2(x-\frac 12)^2,  &x\in[0,\frac 12) \\
    \frac 12-2(x-1)^2,  &x\in[\frac 12,1)
    \end{cases} .
    \end{equation*}
    For any $x,y \in S^1$, $F(x,y)=0$, so $(S^1,T)$ is weak mean equicontinuous. But $0$ and $\frac 12$ are not mean equicontinuous points, which shows $(S^1,T)$ is not mean equicontinuous.
    
    On the one hand, $F$-continuity implies $\overline{F}$-continuity. On the other hand, we can  prove in an $\overline{F}$-continuous system $(X,T)$, all the points are generic points.
    Combining this with Theorem~\ref{1.2}, we deduce that an $\overline{F}$-continuous t.d.s is $F$-continuous. Hence, $F$-continuity is equivalent to $\overline{F}$-continuity. We conclude the relations as follows:
    
    \medskip
	\centerline{equicontinuity $\Rightarrow$ mean equicontinuity $\Rightarrow$ weak mean equicontinuity $\Leftrightarrow$ $F$-continuity.}
	\medskip
	
	We say a system is chaotic if the positions of points in orbits are sensitive to initial values. While weak mean equicontinuous systems may be chaotic, but in the view of measure theory, they are stable, for the distributions of points in orbits are not sensitive to initial values.
	
	Birkhoff Ergodic Theorem shows that for any integrable function $f$, the time average $f^*$
	is also integrable. It is natural to ask in which case the time average operator can preserve continuity of observe functions? In~\cite{JA}, Auslander shows in a mean-L-stable system, the time average operator maps continuous functions to continuous ones. In our paper, we will show that weak mean equicontinuous systems are exact the systems in which the time average operator can preserve continuity of observe functions.
	\begin{theorem}\label{1.5}
		Let $(X,T)$ be a t.d.s. Then $(X,T)$ is weak mean equicontinuous if and only if the time averages $f^{*}$ are continuous for all $f \in C(X)$.
	\end{theorem}

   We organize this paper as follows. In Section \ref{2} we introduce some basic notions and results needed in the paper. In Section \ref{3} we show some propositions of $\overline{F}(x,y)$ and $N(F)$ which are useful in the sequel. In Section \ref{4} we prove Theorem \ref{1.2}. In Section \ref{5} we study invariant measures by $\overline{F}(x,y)$ and $F(x,y)$, and prove Theorem \ref{1.3}. Meanwhile new characterizations of ergodic measures and physical measures are given. In Section \ref{6} we introduce weak mean equicontinuous systems and provide the proof of Theorem \ref{1.5}.
	
   \medskip
   
   \noindent{\bf Acknowledgments.} We would like to express our deep gratitude to Professor Wen Huang for his valuable comments and suggestions. We also thank Weisheng Wu and Qianying Xiao very much for their valuable advice.
		
	\section{Preliminaries}\label{2}
	In this section we recall some notions and results of topological dynamical systems which are needed in our paper. Note that $\mathbb{N}$ denotes the set of all non-negative integers and $\mathbb{N}^+$ denotes the set of all positive integers in this paper.
	\subsection{ Density}
	Let $F \subset \mathbb{N}$, we define the upper density $\overline{D}(F)$ of $F$ by 
	\[
	\overline{D}(F)=\limsup_{n \to +\infty} \frac {\#(F\cap[0,n-1])}{n} ,
	\]
	where $\#(\cdot)$ is the number of elements of a set. Similarly, the lower density $\underline{D}(F)$ of $F$ is defined by
	\[
	\underline{D}(F)=\liminf_{n \to +\infty} \frac {\#(F\cap[0,n-1])}{n} .
	\]
	One may say $F$ has density $D(F)$ if $\overline{D}(F)=\underline{D}(F)$, in which case $D(F)$ is equal to this common value.
	
	\subsection{Invariant measures}
	Suppose $(X,T)$ is a t.d.s. The $\sigma$-algebra of Borel subsets of $X$ will be denoted by $\mathscr{B}(X)$. Let $M(X)$ be the collection of all regular Borel probability measures defined on the measurable space $\big(X, \mathscr{B}(X) \big)$. In the weak$^*$ topology, $M(X)$ is a nonempty compact set (see for example \cite{PW}, Theorem 6.5). 
	
	We say $\mu \in M(X)$ is $T$-invariant if $\mu \big(T^{-1} (A) \big) = \mu (A)$ holds for any $A \in \mathscr{B}(X)$. Denote by $M(X,T)$ the collection of all $T$-invarant regular Borel probability measures defined on the measurable space $\big(X, \mathscr{B}(X) \big)$. In the weak$^*$ topology, $M(X,T)$ is a nonempty compact convex set (see for example \cite{PW}, Corollary 6.9.1, Theorem 6.10).
	
	We say $\mu \in M(X,T)$ is ergodic if for any $A \in \mathscr{B}(X)$ with $T^{-1} A = A$, $\mu (A) = 0$ or $\mu(A) = 1$ holds. Denote by $M^e (X,T)$ the collection of all ergodic measures on $(X,T)$. It is well known that $M^e (X,T)$ is the collection of all extreme points of $M(X,T)$ (see for example \cite{PW}, Theorem 6.10). Thus, $M^e (X,T)$ is nonempty.
	
	We say $(X,T)$ is uniquely ergodic if $M^e (X,T)$ is singleton. Since $M^e (X,T)$ is the set of extreme points of $M(X,T)$, then $(X,T)$ is unique ergodic if and only if $M(X,T)$ is singleton. 

    Given $x \in X$, we have $\{\frac 1n\sum\limits_{k=1}^n \delta_{T^k x}\}_{n=1}^{\infty} \subset M(X)$, where $\delta_y$ is the Dirac measure supported on $y$. Denote by $M_x$ the collection of all limit points of $\{\frac 1n\sum\limits_{k=1}^n \delta_{T^k x}\}_{n=1}^{\infty}$. Since $M(X)$ is compact, we have $M_x \not = \O$. Moreover, $M_x \subset M(X,T)$. We call $M_x$ the measure set generated by $x$.
		
	A point $x \in X$ is called generic point if for any $f \in C(X)$, the time average
	\[
	f^{*}(x)= \lim_{n \to +\infty}\frac 1n \sum_{k=1}^n f(T^k x)
	\]
	exists. It is easy to derive that $x$ is a generic point if and only if $M_x$ consists of a single measure. We call $\mu \in M_x$ is generated by $x$ if $x$ is a generic point. It is well known that when $\mu$ is an ergodic measure, there is a generic point $x \in X$ such that $\mu$ is generated by $x$ (see for example \cite{PW}, Lemma 6.13). We call a generic point $x$ is an ergodic point if the invariant measure generated by $x$ is ergodic.
	
	A Borel subset $E \subset X$ is said to have invariant measure one if $\mu(E)=1$ for all $\mu \in M(X,T)$. Let $Q$ denote the set of all generic points, and $Q_T$ denote the set of all ergodic points. In~\cite{JCO}, Oxtoby proved that $Q$ and $Q_T$ are both Borel sets of invariant measure one.
	
	Next, we define physical measures in a general way.
	\begin{defn}
		Let $(X,T)$ be a t.d.s. and $m \in M(X)$. We say $\mu \in M(X,T)$ is a physical measure with respect to $m$ if $m \big(B(\mu) \big)>0$, where 
		\[
		\text{B}(\mu)=\{x \in X| \lim_{n \to +\infty} \frac 1n \sum_{k=1}^n \delta_{T^k x} = \mu \}.
		\]
	\end{defn}
	
	For any $A \subset X$,
	let
	\begin{equation*} 
	\chi_A(x) = \begin{cases}
	1, & x \in A \\
	0, &x \notin A
	\end{cases} .
	\end{equation*}
   The  following Lemma is well known (see for example \cite{PW}, Remarks on page:149).
	\begin{lemma}\label{lemma2.1}
		 Let $(X,T)$ be a t.d.s. If $x \in X$ is a generic point and $\mu$ is generated by $x$, then for any open set $U \subset X$ and any closed set $V \subset X$, we have 
		 \begin{equation*}
		 \liminf_{n \to +\infty} \frac1n \sum_{k=1}^n \chi_U(T^k x)  \geq \mu(U) \ \ \text{and} \ \ \limsup_{n \to +\infty} \frac1n \sum_{k=1}^n \chi_V(T^k x)  \leq \mu(V).
		 \end{equation*}
	\end{lemma}

   Given $\mu \in M(X)$. Since $X$ is compact, there are finite mutually disjoint subsets of $X$ such that the diameter of each subset is small enough and the sum of their measures are closed enough to one. Hence we have the following result.
   \begin{lemma}\label{measure lemma}
   	Let $\mu \in M(X)$. Then for any $\varepsilon >0$, there are finite mutually disjoint closed sets $\{\Lambda_k\}_{k=1}^{k_0}$ such that 
   	\begin{equation*}
   	\mu(\bigcup_{k=1}^{k_0} \Lambda_k) \geq 1- \varepsilon \ \ \text{and} \ \ diam(\Lambda_k) \leq \varepsilon, {\forall}k=1,2,\cdots,k_0.
   	\end{equation*}
   	Similarly, there are finite mutually disjoint open sets $\{V_s\}_{s=1}^{s_0}$ such that 
   	\begin{equation*}
   	\mu(\bigcup_{s=1}^{s_0} V_s) \geq 1- \varepsilon \ \ \text{and} \ \ diam(V_s) \leq \varepsilon, {\forall}s=1,2,\cdots,s_0.
   	\end{equation*}
   \end{lemma}
	
	\section{Some properties of $\overline{F}(x,y)$ and $N(F)$}\label{3}
	In this section, we will show some properties of $\overline{F}(x,y)$ and $N(F)$, which play important roles in the next sections.
	\begin{prop}\label{P3.1}
		Let $(X,T)$ be a t.d.s. Then
		    \begin{itemize}
		    \item [(\emph{1})] 
			For any sequences $\{x_k \}_{k=1}^n$ and $\{y_k \}_{k=1}^n$ of $X$, we have
			\begin{equation*}
			\inf\limits_{\sigma \in S_n} \sum\limits_{k=1}^n d(x_k, y_{\sigma(k)}) = \inf\limits_{\sigma \in S_n} \sum\limits_{k=1}^n d(y_k, x_{\sigma(k)}) .
			\end{equation*}
			In particular, for any $x,y \in X$, we have
			\begin{equation*}
			\inf_{\sigma \in S_n} \frac 1n \sum\limits_{k=1}^n d(T^k x, T^{\sigma(k)} y) = \inf_{\sigma \in S_n} \frac 1n \sum\limits_{k=1}^n d(T^k y, T^{\sigma(k)} x).
			\end{equation*}
			\item [(\emph{2})]
			For any sequences $\{x_k \}_{k=1}^n$, $\{y_k \}_{k=1}^n$ and $\{z_k \}_{k=1}^n$ of $X$,
			we have
			\begin{equation*}
			\inf\limits_{\sigma \in S_n} \sum\limits_{k=1}^n d(x_k, z_{\sigma(k)}) \le \inf\limits_{\sigma \in S_n} \sum\limits_{k=1}^n d(x_k, y_{\sigma(k)}) + \inf\limits_{\sigma \in S_n} \sum\limits_{k=1}^n d(y_k, z_{\sigma(k)}). 
			\end{equation*}
			In particular, for any $x,y,z \in X$, we have
			\begin{equation*}
			\inf_{\sigma \in S_n} \frac 1n \sum\limits_{k=1}^n d(T^k x, T^{\sigma(k)} z) \le \inf_{\sigma \in S_n} \frac 1n \sum\limits_{k=1}^n d(T^k x, T^{\sigma(k)} y)  + \inf_{\sigma \in S_n} \frac 1n \sum\limits_{k=1}^n d(T^k y, T^{\sigma(k)} z).
			\end{equation*}
			\item [(\emph{3})]
			For any $x,y \in X$, we have
			\begin{equation*}
			\overline{F}(x,y) = \overline{F}(y,x)
			\end{equation*}
			and
			\begin{equation*}
			\underline{F}(x,y) = \underline{F}(y,x).
			\end{equation*}
			
			\item [(\emph{4})]
			For any $x,y,z \in X$, we have
			\begin{equation*}
			\overline{F}(x,z) \le \overline{F}(x,y) + \overline{F}(y,z)
			\end{equation*}
			and
			\begin{equation*}
			\underline{F}(x,z) \le \underline{F}(x,y) + \overline{F}(y,z).
			\end{equation*}
		\end{itemize}
	\end{prop}
	\begin{proof}
		(1) There exists a $\sigma_1 \in S_n$ such that 
		\begin{equation*}
		\sum\limits_{k=1}^n d(x_k, y_{\sigma_1 (k)}) = \inf\limits_{\sigma \in S_n} \sum\limits_{k=1}^n d(x_k, y_{\sigma(k)}).
		\end{equation*}
		Let $\sigma_2 \in S_n$ such that $\sigma_1 \sigma_2 = \sigma_2 \sigma_1$ be the indentity element of $S_n$. Then we have
		\begin{equation*}
		\sum\limits_{k=1}^n d(x_k, y_{\sigma_1 (k)}) = \sum\limits_{k=1}^n d(x_{\sigma_2 \sigma_1 (k)}, y_{\sigma_1 (k)}) = \sum\limits_{k=1}^n d(x_{\sigma_2 (k)}, y_k) \ge \inf\limits_{\sigma \in S_n} \sum\limits_{k=1}^n d(y_k, x_{\sigma(k)}). 
		\end{equation*}
		Thus, 
		\begin{equation*}
		\inf\limits_{\sigma \in S_n} \sum\limits_{k=1}^n d(x_k, y_{\sigma(k)}) \ge \inf\limits_{\sigma \in S_n} \sum\limits_{k=1}^n d(y_k, x_{\sigma(k)}).
		\end{equation*}
		With the same reason, we can get
		\begin{equation*}
		\inf\limits_{\sigma \in S_n} \sum\limits_{k=1}^n d(y_k, x_{\sigma(k)}) \ge \inf\limits_{\sigma \in S_n} \sum\limits_{k=1}^n d(x_k, y_{\sigma(k)}).
		\end{equation*}
		Hence,
		\begin{equation*}
		\inf\limits_{\sigma \in S_n} \sum\limits_{k=1}^n d(x_k, y_{\sigma(k)}) = \inf\limits_{\sigma \in S_n} \sum\limits_{k=1}^n d(y_k, x_{\sigma(k)}).
		\end{equation*}
		
		(2) There are $\sigma_1, \sigma_2 \in S_n$ such that
		\begin{equation*}
		\sum\limits_{k=1}^n d(x_k, y_{\sigma_1 (k)}) = \inf_{\sigma \in S_n} \sum\limits_{k=1}^n d(x_k, y_{\sigma (k)})
		\end{equation*}
		and
		\begin{equation*}
		\sum\limits_{k=1}^n d(y_k, z_{\sigma_2 (k)}) = \inf_{\sigma \in S_n} \sum\limits_{k=1}^n d(y_k, z_{\sigma (k)}).
		\end{equation*}
		Let $\sigma_3 = \sigma_2 \sigma_1$, then we have
		\begin{equation*}
		\begin{split}
		\sum\limits_{k=1}^n d(x_k, z_{\sigma_3 (k)})
		& \le \sum\limits_{k=1}^n d(x_k, y_{\sigma_1 (k)}) + \sum\limits_{k=1}^n d(y_{\sigma_1 (k)}, z_{\sigma_3 (k)}) \\
		& = \sum\limits_{k=1}^n d(x_k, y_{\sigma_1 (k)}) + \sum\limits_{k=1}^n d(y_k, z_{\sigma_2 (k)}) \\
		&= \inf_{\sigma \in S_n} \sum\limits_{k=1}^n d(x_k, y_{\sigma (k)}) + \inf_{\sigma \in S_n} \sum\limits_{k=1}^n d(y_k, z_{\sigma (k)}).
		\end{split}
		\end{equation*}
		Thus,
		\begin{equation*}
		\inf_{\sigma \in S_n} \sum\limits_{k=1}^n d(x_k, z_{\sigma (k)}) \le \inf_{\sigma \in S_n} \sum\limits_{k=1}^n d(x_k, y_{\sigma (k)}) + \inf_{\sigma \in S_n} \sum\limits_{k=1}^n d(y_k, z_{\sigma (k)}).
		\end{equation*}
		
		By (1) and (2), we can easily deduce (3) and (4).		
	\end{proof}
	
    $\overline{F}(x,y)$ and $\underline{F}(x,y)$ are functions which can measure the difference between distributions of $\text{Orb}(x)$ and $\text{Orb}(y)$. When $x$ and $y$ are in the same orbit, the distributions of $\text{Orb}(x)$ and $\text{Orb}(y)$ are same. Thus, $F(x,y)=0$. Next, we shall show that $\overline{F}|_{\text{Orb}(x) \times \text{Orb}(y)}$ and $\underline{F}|_{\text{Orb}(x) \times \text{Orb}(y)}$ are constants for any $x,y \in X$.
    \begin{prop}\label{P3.2}
    	Let $(X,T)$ be a t.d.s. For any $x,y \in X$ and $r, s \in \mathbb{N}$, we have
    	\[
    	\overline{F}(T^r x,T^s y) =\overline{F}(x,y)
    	\]
    	and
    	\[
    	\underline{F}(T^r x,T^s y) =\underline{F}(x,y).
    	\]
    	If $F(x,y)$ exists, we also have
    	\[
    	F(T^r x,T^s y) =F(x,y).
    	\]
    \end{prop} 
    \begin{proof}
    	By Proposition~\ref{P3.1}, we have
    	\[
    	\underline{F}(T^r x,y) \le \underline{F}(x,y) + \overline{F}(T^r x,x)= \underline{F}(x,y) .
    	\]
    	On the other hand, we have
    	\[
    	\underline{F}(x,y) \le \underline{F}(T^r x,y) + \overline{F}(x,T^r x) = \underline{F}(T^r x,y).
    	\]
    	Thus,
    	\[
    	\underline{F}(T^r x,y) = \underline{F}(x,y).
    	\]
    	Similarly, we can deduce that
    	\[
    	\underline{F}(T^r x,T^s y) = \underline{F}(T^r x,y).
    	\]
    	Hence, we have
    	\[
        \underline{F}(T^r x,T^s y)=\underline{F}(x,y).
    	\]
    	With the same reason, we can deduce the last two equalities.
    \end{proof}

    By Proposition~\ref{P3.2}, we can deduce $N(F)$ and $N(\underline{F})$ are both invariant sets with respect to $T^r \times T^s$ for any $r, s \in \mathbb{N}$. Given $x \in X$, let
    \[
    N(F,x)=\{y\in X | F(x,y)=0\}
    \]
    and
   \[
   N(\underline{F},x)=\{y\in X | \underline{F}(x,y)=0\}.
   \]
   Then we have the following proposition.
   \begin{prop}\label{P3.3}
   	Let $(X,T)$ be a t.d.s. Then for any $x\in X$, $N(F,x)$ and $N(\underline{F},x)$ are both Borel invariant sets. 
   \end{prop}
   \begin{proof}
   	For any $m, n \in \mathbb{N^+}$ and any $\sigma \in S_n$, let
   	\[
   	O (x, F_{n,\sigma}, \frac 1m)= \{y \in X| F_{n,\sigma}(x,y) < \frac 1m \},
   	\]
   	where 
   	\[
   	F_{n,\sigma}(x,y) = \frac 1n \sum_{k=1}^n d(T^k x, T^{\sigma(k)} y).
   	\]
   	Then $O (x, F_{n,\sigma}, \frac 1m)$ is an open set. 
   	Since
   	\[
   	N(F,x)=\bigcap_{m=1}^{\infty} \bigcup_{r=1}^{\infty} \bigcap_{n=r}^{\infty} \bigcup_{\sigma \in S_n} O (x, F_{n,\sigma}, \frac 1m)
   	\]
   	and
   	\[
   	N(\underline{F},x)=\bigcap_{m=1}^{\infty} \bigcap_{r=1}^{\infty} \bigcup_{n=r}^{\infty} \bigcup_{\sigma \in S_n} O (x, F_{n,\sigma}, \frac 1m),
   	\]
   we derive that $N(F,x)$ and $N(\underline{F},x)$ are both Borel sets.
   	 
   	On the other hand, by Proposition~\ref{P3.2} we derive that $N(F,x)$ and $N(\underline{F},x)$ are both invariant sets.
   \end{proof}
 
   The following proposition provides a way to estimate the upper bound of $\overline{F}(x,y)$. 
   
   \begin{prop}\label{P3.4}
   	Let $(X,T)$ be a t.d.s. and $\{U_s \}_{s=1}^{s_0}$ be mutually disjoint subsets of $X$. Given $x,y \in X$. If the following two conditions hold:
   	\begin{itemize}
   		\item [(\emph{1})]
   		There is $\varepsilon > 0$ such that $diam(U_s) \le \varepsilon$ holds for any $s \in \{1, 2, \cdots, s_0 \}$;
   		\item [(\emph{2})]
   		For any $s \in \{1, 2, \cdots, s_0 \}$, there is $a_s \ge 0$ such that
   		\begin{equation*}
   		\liminf_{n \to +\infty} \frac 1n \sum\limits_{k=1}^n \chi_{U_{s}} (T^k x) \ge a_s
   		\end{equation*}
   		and 
   		\begin{equation*}
   		\liminf_{n \to +\infty} \frac 1n \sum\limits_{k=1}^n \chi_{U_{s}} (T^k y) \ge a_s.
   		\end{equation*}
   	\end{itemize}
   	Then we have 
   	\begin{equation*}
   	\overline{F}(x,y) \le \varepsilon \sum\limits_{s=1}^{s_0} a_s + M(1- \sum\limits_{s=1}^{s_0} a_s),
   	\end{equation*}
   	where $M = diam(X)$.
   \end{prop}

   \begin{proof}
   	Given $\delta > 0$, there is an $n_1 > 0$ such that for any $n> n_1$ and any $s \in \{1, 2, \cdots, s_0 \}$, we have
   	\begin{equation}\label{eq3.1}
   	\frac1{n} \sum_{k=1}^n \chi_{U_s} (T^k x)  \geq \liminf_{n \to +\infty} \frac 1n \sum\limits_{k=1}^n \chi_{U_{s}} (T^k x) - \delta \ge a_s - \delta.
   	\end{equation}
   	Similarly, there is an $n_2 > 0$ such that for any $n> n_2$ and any $s \in \{1, 2, \cdots, s_0 \}$, we have
    \begin{equation}\label{eq3.2}
    	\frac1{n} \sum_{k=1}^n \chi_{U_s} (T^k y) \ge a_s - \delta.
    \end{equation}
   
   	Given $n > \max \{n_1, n_2\}$. For any $s \in \{1, 2, \cdots, s_0 \}$, let
   	\[
   	N_n(U_s,x)=\{k\in \mathbb{N^{+}}| T^k x \in U_s, k \le n \} \ \ \text{and} \ \ N_n(U_s,y)=\{k\in \mathbb{N^{+}}| T^k y \in U_s, k \le n \}.
   	\]
   	By \eqref{eq3.1} and \eqref{eq3.2}, there exists $r_{s,n} \in \mathbb{N}$ with 
   	\begin{equation}\label{eq3.3}
   	a_s n +1 \ge r_{s,n} \ge (a_s - \delta)n
   	\end{equation}
   	such that 
   	\begin{equation*}
   	\# \big(N_n(U_s,x) \big) \ge r_{s,n} \ , \  \# \big(N_n(U_s,y) \big) \ge r_{s,n} \ .
   	\end{equation*}
   	Thus there are subsets $\{k_{x,s,r}\}_{r=1}^{r_{s,n}}$ and $\{k_{y,s,r}\}_{r=1}^{r_{s,n}}$ of $\mathbb{N}^+$ such that
   	\begin{equation*}
   	\{k_{x,s,1}, k_{x,s,2}, \cdots, k_{x,s,r_{s,n}}\} \subset N_n(U_s,x)
   	\end{equation*}
   	and
   	\begin{equation*}
   	\{k_{y,s,1}, k_{y,s,2}, \cdots, k_{y,s,r_{s,n}}\} \subset N_n(U_s,y).
   	\end{equation*} 
   	
   	Since $N_n(U_{s_1},x) \bigcap N_n(U_{s_2},x) = \O$ and $N_n(U_{s_1},y) \bigcap N_n(U_{s_2},y) = \O$ hold for any \\
   	$s_1, s_2 \in\{1,2, \cdots, s_0\}$ with $s_1 \not= s_2$, there is 
   	$\sigma_n \in S_n$ such that 
   	\begin{equation}\label{eq3.4}
   	\sigma_n (k_{x,s,r}) = k_{y,s,r}
   	\end{equation}
   	holds for any $s\in \{1,2, \cdots, s_0\}$ and any $r \in \{1, 2, \cdots, r_{s,n}\}$.
   	
   	Let
   	\begin{equation*}
   	A= \bigcup_{s=1}^{s_0} \{k_{x,s,r}\}_{r=1}^{r_{s,n}} \ \ \text{and} \ \ 	B= \{1, 2, \cdots, n\} \backslash A.
   	\end{equation*}
   	Then by \eqref{eq3.3}, we have
   	\begin{equation}\label{eq3.5}
   	\#(A) = \sum_{s=1}^{s_0} r_{s,n} \le \sum_{s=1}^{s_0} (a_s n + 1)
   	\end{equation}
   	and
   	\begin{equation}\label{eq3.6}
   	\#(B) = n - \sum\limits_{s=1}^{s_0} r_{s,n} \le n - \sum_{s=1}^{s_0} (a_s - \delta)n.
   	\end{equation}
   	
   	By \eqref{eq3.4}, we deduce that 
   	\begin{equation}\label{eq3.7}
   	d(T^k x, T^{\sigma_n(k)} y) \le \varepsilon
   	\end{equation}
   	holds for any $k \in A$. On the other hand, for any $k \in B$ we have
   	\begin{equation}\label{eq3.8}
   	d(T^k x, T^{\sigma_n(k)} y) \le diam(X) = M.
   	\end{equation}
   	Hence, we derive that
   	\begin{equation*}
   	\begin{split}
   	\inf_{\sigma \in S_n} \frac 1{n} \sum_{k=1}^{n} d(T^k x,T^{\sigma(k)} y) 
   	& \le \frac 1{n} \sum_{k=1}^{n} d(T^k x,T^{\sigma_n(k)} y) \\
   	& = \frac 1n \sum_{k \in A} d(T^k x,T^{\sigma_n(k)} y) + \frac 1n \sum_{k \in B} d(T^k x,T^{\sigma_n(k)} y) \\
   	& \le \frac {\varepsilon}n \#(A) + \frac Mn \#(B) \ \ \big({\text by \ \eqref{eq3.7} , \eqref{eq3.8}  }\big) \\
   	& \le \frac {\varepsilon}n \sum_{s=1}^{s_0} (a_s n + 1) + \frac Mn \big(n -\sum_{s=1}^{s_0} (a_s - \delta)n \big),
   	\end{split}
   	\end{equation*}
   	where the last inequality comes from \eqref{eq3.5} and \eqref{eq3.6}. Let $n \to +\infty$, then we have
   	\begin{equation*}
   	\overline{F}(x,y) = \limsup_{n \to +\infty} \inf_{\sigma \in S_n} \frac 1{n} \sum_{k=1}^{n} d(T^k x,T^{\sigma(k)} y) \le \varepsilon \sum_{s=1}^{s_0} a_s + M \big(1- \sum_{s=1}^{s_0} (a_s - \delta) \big).
   	\end{equation*}
   	Let $\delta \to 0$, and then we deduce that
   	\begin{equation*}
   	\overline{F}(x,y) \le \varepsilon \sum_{s=1}^{s_0} a_s + M (1- \sum_{s=1}^{s_0} a_s ).
   	\end{equation*}
   	This finishes the proof of Proposition \ref{P3.4}.
   \end{proof}
  
  With respect to $\underline{F}(x,y)$, we have the similar proposition.
  
   \begin{prop}
  	Let $(X,T)$ be a t.d.s. and $\{U_s \}_{s=1}^{s_0}$ be mutually disjoint subsets of $X$. Given $x,y \in X$. If the following two conditions hold:
  	\begin{itemize}
  		\item [(\emph{1})]
  		There is $\varepsilon > 0$ such that $diam(U_s) \le \varepsilon$ holds for any $s \in \{1, 2, \cdots, s_0 \}$;
  		\item [(\emph{2})]
  		There is a subsequence $\{n_r\}_{r=1}^{\infty}$ of $\mathbb{N}^+$ such that for any $s \in \{1, 2, \cdots, s_0 \}$, the following inequalities 
  		\begin{equation*}
  		\lim_{r \to +\infty} \frac 1{n_r} \sum\limits_{k=1}^{n_r} \chi_{U_{s}} (T^k x) \ge a_s
  		\end{equation*}
  		and 
  		\begin{equation*}
  		\lim_{r \to +\infty} \frac 1{n_r} \sum\limits_{k=1}^{n_r} \chi_{U_{s}} (T^k y) \ge a_s.
  		\end{equation*}
  		hold for some $a_s \ge 0$.
  	\end{itemize}
  	Then we have 
  	\begin{equation*}
  	\underline{F}(x,y) \le \varepsilon \sum\limits_{s=1}^{s_0} a_s + M(1- \sum\limits_{s=1}^{s_0} a_s),
  	\end{equation*}
  	where $M = diam(X)$.
  \end{prop}

   \section{Proof of Theorem \ref{1.2}}\label{4}
   In this section, we will prove Theorem \ref{1.2}. Assume the contrary that there are generic points $x,y \in X$ such that $F(x,y)$ does not exist, this is $\alpha = \overline{F}(x,y) -  \underline{F} (x,y) > 0$. Then we estimate the upper bound of $\overline{F}(x,y)$ and the lower bound of $\underline{F} (x,y)$, from which we deduce that $\overline{F}(x,y) -  \underline{F} (x,y) \le \frac {\alpha}2$. This contradicts with the assumption. So $F(x,y)$ exists when $x,y \in X$ are generic points. In the Proof of Theorem \ref{1.2}, we need the following lemma which is a direct corollary of Birkhoff-Von Neumann Theorem \cite{AB}.
   
     \begin{lemma}\label{BN_lemma}
     	Let $X$ be a metric space and $m, n \in \mathbb{N}^+$. If $\{x_i\}_{i=1}^n, \{y_i\}_{i=1}^n$, $\{\overline{x}_j \}_{j=1}^{mn}$ and $\{\overline{y}_j \}_{j=1}^{mn}$ are subsequences of $X$ and there is $\sigma_{mn} \in S_{mn}$ such that
     	\begin{equation*}
     	\overline{x}_{\sigma_{mn} (j)} = x_i, \ \forall j = m(i-1)+1, m(i-1)+2, \cdots, m(i-1) + m
     	\end{equation*}
     	and
     	\begin{equation*}
     	\overline{y}_{\sigma_{mn} (j)} = y_i, \ \forall j = m(i-1)+1, m(i-1)+2, \cdots, m(i-1) + m
     	\end{equation*}
     	hold for any $i \in \{ 1, 2, \cdots, n \}$, then we have
     \begin{equation*}
     \inf_{\sigma \in S_{mn} } \sum_{j=1}^{mn} d(\overline{x}_j, \overline{y}_{\sigma(j)}) = m \cdot \inf_{\sigma \in S_n} \sum_{i=1}^n d( x_i, y_{\sigma(i)}).
     \end{equation*}
     \end{lemma}
   
    Next, we provide the detailed proof of Theorem \ref{1.2}.
	\begin{proof}[\bf Proof of Theorem~\ref{1.2}]
		Let $x,y \in X$ be generic points of $(X,T)$. Put
		\[
		\alpha = \overline{F}(x,y)-\underline{F}(x,y).
		\]
		We assume that $F(x,y)$ does not exist, then $\alpha > 0$.
		
		Since $x,y$ are generic points, there are $\mu_x, \mu_y \in M(X,T)$ such that
		\begin{equation*}
		\mu_x = \lim_{n \to +\infty} \frac 1n \sum_{k=1}^n \delta_{T^k x} \ \ \text{and} \ \ \mu_y = \lim_{n \to +\infty} \frac 1n \sum_{k=1}^n \delta_{T^k y}.
		\end{equation*}
		Let $\varepsilon = \frac {\alpha}{8+16M}$, where $M=diam(X)$. By Lemma~\ref{measure lemma}, there exist finite mutually disjoint  open sets $\{\Lambda_r\}_{r=1}^{r_0}$ of $X$ such that 
		\begin{equation}\label{eq4.1}
		\sum_{r=1}^{r_0}\mu_x(\Lambda_r) \ge 1 - \varepsilon, \ \  diam(\Lambda_r) \le \varepsilon, {\forall}r=1,2,\cdots,r_0.
		\end{equation}
		Similarly, there exist finite mutually disjoint open sets $\{V_s\}_{s=1}^{s_0}$ of $X$ 
		such that
		\begin{equation}\label{eq4.2}
		\sum_{s=1}^{s_0}\mu_y(V_s) \ge 1 - \varepsilon, \ \ diam(V_s) \le \varepsilon, {\forall}s=1,2,\cdots,s_0.
		\end{equation}
		Without loss of generality, we can assume that $\mu_x(\Lambda_r)>0$ and $\mu_y(V_s)>0$ for any $r \in \{1,2,\cdots,r_0 \}$, $s \in \{1,2,\cdots,s_0\}$. Let $\Lambda_{r_0+1} = X \backslash \bigcup\limits_{r=1}^{r_0} \Lambda_r$ and $V_{s_0+1} = X \backslash \bigcup\limits_{s=1}^{s_0} V_s$.
		We select sequences $\{x_r\}_{r=1}^{r_0 +1}$ and $\{y_s\}_{s=1}^{s_0 +1}$ of $X$ such that $x_r \in \Lambda_r$ and $y_s \in V_s$ for any  $r \in \{1,2,\cdots,r_0 +1\}, s \in \{1,2,\cdots,s_0 +1\}$.
		
		Let $\{\hat{x}_k \}_{k=1}^n$ and $\{\hat{y}_k \}_{k=1}^n$ be sequences such that $\hat{x}_k = x_r$ if $T^k x \in \Lambda_r$ and $\hat{y}_k = y_s$ if $T^k y \in V_s$.
		By Proposition \ref{P3.1}, we have 
		\begin{equation}\label{eq4.3}
		\begin{split}
		&\hskip 0.5cm \inf_{\sigma \in S_n} \frac 1n \sum_{k=1}^n d(T^k x,T^{\sigma(k)} y)\\
		&\le \inf_{\sigma \in S_n} \frac 1n \sum_{k=1}^n d (T^{k} x,\hat{x}_{\sigma(k)}) + \inf_{\sigma \in S_n} \frac 1n \sum_{k=1}^n d(\hat{x}_{\sigma(k)}, \hat{y}_k)+ \inf_{\sigma \in S_n} \frac 1n \sum_{k=1}^n d (\hat{y}_k, T^{\sigma(k)} y)\\
		&= \inf_{\sigma \in S_n} \frac 1n \sum_{k=1}^n d (\hat{x}_k, T^{\sigma(k)} x) + \inf_{\sigma \in S_n} \frac 1n \sum_{k=1}^n d(\hat{x}_k, \hat{y}_{\sigma(k)} )+ \inf_{\sigma \in S_n} \frac 1n \sum_{k=1}^n d (\hat{y}_k, T^{\sigma(k)} y).
		\end{split}
		\end{equation}
		Similarly, we have
		\begin{equation*}
		\begin{split}
		&\hskip 0.5cm \inf_{\sigma \in S_n} \frac 1n \sum_{k=1}^n d(\hat{x}_k, \hat{y}_{\sigma(k)} )\\
		&\le \inf_{\sigma \in S_n} \frac 1n \sum_{k=1}^n d (\hat{x}_k, T^{\sigma(k)} x) + \inf_{\sigma \in S_n} \frac 1n \sum_{k=1}^n d(T^{\sigma(k)} x,T^k y) + \inf_{\sigma \in S_n} \frac 1n \sum_{k=1}^n d (T^k y, \hat{y}_{\sigma(k)} )\\
		&= \inf_{\sigma \in S_n} \frac 1n \sum_{k=1}^n d (\hat{x}_k, T^{\sigma(k)} x) + \inf_{\sigma \in S_n} \frac 1n \sum_{k=1}^n d(T^k x,T^{\sigma(k)} y) + \inf_{\sigma \in S_n} \frac 1n \sum_{k=1}^n d (\hat{y}_k, T^{\sigma(k)} y).
		\end{split}
		\end{equation*}
		Thus, we deduce that 
		\begin{equation}\label{eq4.4}
		\begin{split}
		&\hskip 0.5cm \inf_{\sigma \in S_n} \frac 1n \sum_{k=1}^n d(T^k x,T^{\sigma(k)} y)\\
		&\ge \inf_{\sigma \in S_n} \frac 1n \sum_{k=1}^n d(\hat{x}_k, \hat{y}_{\sigma(k)} ) - \inf_{\sigma \in S_n} \frac 1n \sum_{k=1}^n d (\hat{x}_k, T^{\sigma(k)} x) - \inf_{\sigma \in S_n} \frac 1n \sum_{k=1}^n d (\hat{y}_k, T^{\sigma(k)} y).
		\end{split}
		\end{equation}
		
		Our aim is to estimate the bounds of $\inf\limits_{\sigma \in S_n} \frac 1n \sum\limits_{k=1}^n d(T^k x,T^{\sigma(k)} y)$. And inequalities \eqref{eq4.3} and \eqref{eq4.4} show that to achieve this aim it suffices to estimate the bounds of $\inf\limits_{\sigma \in S_n} \frac 1n \sum\limits_{k=1}^n d(\hat{x}_k, \hat{y}_{\sigma(k)} )$,\\
		$\inf\limits_{\sigma \in S_n} \frac 1n \sum\limits_{k=1}^n d (\hat{x}_k, T^{\sigma(k)} x)$ and $\inf\limits_{\sigma \in S_n} \frac 1n \sum\limits_{k=1}^n d (\hat{y}_k, T^{\sigma(k)} y)$. In the following, Lemma \ref{lemma4.1} shows the upper bounds of $\inf\limits_{\sigma \in S_n} \frac 1n \sum\limits_{k=1}^n d (\hat{x}_k, T^{\sigma(k)} x)$ and $\inf\limits_{\sigma \in S_n} \frac 1n \sum\limits_{k=1}^n d (\hat{y}_k, T^{\sigma(k)} y)$, Lemma \ref{lemma4.2} shows the lower bound of $\inf\limits_{\sigma \in S_n} \frac 1n \sum\limits_{k=1}^n d(\hat{x}_k, \hat{y}_{\sigma(k)} )$ and Lemma \ref{lemma4.3} shows the upper bound of $\inf\limits_{\sigma \in S_n} \frac 1n \sum\limits_{k=1}^n d(\hat{x}_k, \hat{y}_{\sigma(k)} )$.
		
		Given $\beta > 0$. By Lemma \ref{lemma2.1}, for all sufficiently large $n \in \mathbb{N}^+$ we have
		 \begin{equation}\label{eq4.5}
		 \frac 1n \sum_{k=1}^n\chi_{\Lambda_r} (T^k x) \ge \mu_x (\Lambda_r) - \beta \ \ \text{and} \ \ \frac 1n \sum_{k=1}^n\chi_{V_s} (T^k y) \ge \mu_y (V_s) - \beta
		 \end{equation}
		 hold for any $r \in \{1, 2, \cdots, r_0\}$ and $s \in \{1, 2, \cdots, s_0\}$.

		\begin{lemma}\label{lemma4.1}
			For all sufficiently large $n \in \mathbb{N}^+$, we have
			\begin{equation*}
			\inf_{\sigma \in S_n} \frac 1n \sum_{k=1}^n d(\hat{x}_k, T^{\sigma(k)} x ) \le \varepsilon + M \varepsilon + M r_0 \beta \ \ \text{and} \ \ \inf_{\sigma \in S_n} \frac 1n \sum_{k=1}^n d(\hat{y}_k, T^{\sigma(k)} y ) \le \varepsilon + M \varepsilon + M s_0 \beta .
			\end{equation*}
		\end{lemma}
	
	  \begin{proof}[Proof of Lemma \ref{lemma4.1}]
	  	By \eqref{eq4.1}, we have $d(\hat{x}_k, T^k x) \le \varepsilon$ if $T^k x \in \bigcup_{r=1}^{r_0} \Lambda_r$. On the other hand, we can estimate $d(\hat{x}_k, T^k x) \le M$ for $diam(X) = M$. Then we have 
	  	\begin{equation}\label{eq4.6}
	  	\begin{split}
	  	& \hskip 0.5cm \inf_{\sigma \in S_n} \frac 1n \sum_{k=1}^n d(\hat{x}_k, T^{\sigma(k)} x ) \le \frac 1n \sum_{k=1}^n d(\hat{x}_k, T^k x ) \\
	  	& \le \varepsilon \cdot \frac {\#(\{k \in \mathbb{N}^+: T^k x \in \bigcup_{r=1}^{r_0} \Lambda_r, k\le n \})}n + M \cdot \frac {\#(\{k \in \mathbb{N}^+: T^k x \notin \bigcup_{r=1}^{r_0} \Lambda_r, k\le n \})}n \\
	  	& \le  \varepsilon + M \big(1- \frac {\#(\{k \in \mathbb{N}^+: T^k x \in \bigcup_{r=1}^{r_0} \Lambda_r, k\le n \})}n \big).	
	  	\end{split}
	  	\end{equation}
	  	Since $\{\Lambda_r \}_{r=1}^{r_0}$ are mutually disjoint, for sufficiently large $n \in \mathbb{N}^+$ we have
	  	\begin{equation}\label{eq4.7}
	  	\begin{split}
	  	\frac {\#(\{k \in \mathbb{N}^+: T^k x \in \bigcup_{r=1}^{r_0} \Lambda_r, k\le n \})}n 
	  	& = \sum_{r=1}^{r_0} \frac {\#(\{k \in \mathbb{N}^+ : T^k x \in \Lambda_r, k \le n\})}n \\
	  	& = \sum_{r=1}^{r_0} \frac 1n \sum_{k=1}^n \chi_{\Lambda_r} (T^k x) \\
	  	& \ge \sum_{r=1}^{r_0} \big( \mu_x (\Lambda_r) - \beta \big) \ \ \big({\text by \ \eqref{eq4.5}} \big) \\
	  	& = \sum_{r=1}^{r_0} \mu_x (\Lambda_r) -  r_0 \beta.	
	  	\end{split}
	  	\end{equation}
	  	Combining \eqref{eq4.6} with \eqref{eq4.7}, we derive that
	  	\begin{equation}\label{eq4.8}
	  	\begin{split}
	  	\inf_{\sigma \in S_n} \frac 1n \sum_{k=1}^n d(\hat{x}_k, T^{\sigma(k)} x ) 
	  	& \le  \varepsilon + M \big(1- \frac {\#(\{k \in \mathbb{N}^+: T^k x \in \bigcup_{r=1}^{r_0} \Lambda_r, k\le n \})}n \big) \\
	  	& \le \varepsilon + M \big(1-  \sum_{r=1}^{r_0} \mu_x (\Lambda_r) + r_0 \beta \big)\\
	  	& \le \varepsilon + M \varepsilon + M r_0 \beta,
	  	\end{split}
	  	\end{equation}
	  	where the last inequality comes from \eqref{eq4.1}.  Similarly, we have
	  	\begin{equation}\label{eq4.9}
	  	\inf_{\sigma \in S_n} \frac 1n \sum_{k=1}^n d(\hat{y}_k, T^{\sigma(k)} y ) \le \varepsilon + M \varepsilon + M s_0 \beta .
	  	\end{equation}
	  	This finishes the proof of Lemma \ref{lemma4.1}.
	  \end{proof}
	
	  To estimate the bounds of $\inf\limits_{\sigma \in S_n} \frac 1n \sum\limits_{k=1}^n d(\hat{x}_k, \hat{y}_{\sigma(k)} )$, we introduce some notations. Let $a_1= \min \{\mu_x(\Lambda_r),r=1,2,\cdots,r_0\}$, $a_2= \min \{\mu_y(V_s),s=1,2,\cdots,s_0\}$ and $a=\min \{\frac {\varepsilon}{r_0+s_0}, a_1, a_2\}$. Then for any $r \in \{1,2,\cdots,r_0\}$ and any $s\in \{1,2,\cdots,s_0\}$, 
	  there are $n_r, m_s \in \mathbb{N^+}$ such that
	  \begin{equation}\label{eq4.10}
	  an_r \le \mu_x(\Lambda_r) <a(n_r+1) 
	  \end{equation}
	  and
	  \begin{equation}\label{eq4.11}
	  am_s \le \mu_y(V_s) <a(m_s+1).
	  \end{equation}

	  Let $K=\min\{\sum\limits_{r=1}^{r_0}n_r,\sum\limits_{s=1}^{s_0}m_s\}$. Without loss of generality, we can assume $K=\sum\limits_{r=1}^{r_0}n_r \le \sum\limits_{s=1}^{s_0}m_s$. Then there is $s_0^* \le s_0$ and integer sequence $\{m_s^*\}_{s=1}^{s_0^*}$ such that 
	  \begin{equation}\label{eq4.12}
	  1 \le m_s^* \le m_s, {\forall}s=1,2,\cdots,s_0^*
	  \end{equation}
	  and 
	  \begin{equation}\label{eq4.13}
	  K=\sum\limits_{s=1}^{s_0^*}m_s^*.
	  \end{equation}
	  By \eqref{eq4.1} we derive that
	  \begin{equation}\label{eq4.14}
	  \begin{split}
	  1- Ka 
	  &\le \varepsilon + \sum_{r=1}^{r_0} \mu_x (\Lambda_r) - a \sum_{r=1}^{r_0} n_r\\
	  &= \varepsilon + \sum_{r=1}^{r_0} \big(\mu_x (\Lambda_r) -a n_r \big)\\
	  &< \varepsilon + r_0 a  \ \ \big({\text by \ \eqref{eq4.10} \big)  } \\
	  &\le 2 \varepsilon,
	  \end{split}
	  \end{equation} 
	  where the last inequality comes from $a \le \frac {\varepsilon}{r_0 + s_0}$. 
	  
	   Construct the sequences $\{\overline{x}_i \}_{i=1}^K$ and $\{\overline{y}_i \}_{i=1}^K$ as follows:
	  \begin{equation*} 
	  \overline{x}_i = \begin{cases}
	  x_1, & i \le n_1\\
	  x_{r+1}, & \sum\limits_{j=1}^r n_j < i \le \sum\limits_{j=1}^{r+1} n_j, r=1, 2, \cdots, r_0 - 1,
	  \end{cases}
	  \end{equation*}
	  \begin{equation*} 
	  \overline{y}_i = \begin{cases}
	  y_1, & i \le m_1^*\\
	  y_{s+1}, & \sum\limits_{j=1}^s m_j^* < i \le \sum\limits_{j=1}^{s+1} m_j^*, s=1, 2, \cdots, s_0^* - 1.
	  \end{cases}
	  \end{equation*}
	  In the following we will use the distance between sequences $\{\overline{x}_i \}_{i=1}^K$ and $\{\overline{y}_i \}_{i=1}^K$ to estimate the bounds of $\inf\limits_{\sigma \in S_n} \frac 1n \sum\limits_{k=1}^n d(\hat{x}_k, \hat{y}_{\sigma(k)} )$. 
	  
	  \begin{lemma}\label{lemma4.2}
	  For all sufficiently large $n \in \mathbb{N}^+$, we have
	  \begin{equation*}
	  \inf\limits_{\sigma \in S_n} \frac 1n \sum\limits_{k=1}^n d(\hat{x}_k, \hat{y}_{\sigma(k)} ) \ge (a-\beta) \cdot \inf\limits_{\sigma \in S_K} \sum\limits_{k=1}^K d(\overline{x}_k, \overline{y}_{\sigma(k)}) - \frac {MK}n- 2M\varepsilon - MK \beta.
	  \end{equation*}
      \end{lemma}
	 
	 \begin{proof}[Proof of Lemma \ref{lemma4.2}]
	 	Let $l_n$ be the minimal integer such that $l_n \ge (a-\beta)n$. Then by \eqref{eq4.14} we have 
	 	\begin{equation} \label{eq4.15}
	 	\frac {n-K l_n}n \le 1- K(a- \beta)  \le 2 \varepsilon +K \beta.
	 	\end{equation}
	 	
	 	{\bf Claim 1.} For all sufficiently large $n \in \mathbb{N}^+$, we have
	 	\begin{equation*}
	 	\#(\{k| \hat{x}_k = x_r, k\le n \}) \ge l_n n_r \ \ \text{and} \ \  \#(\{k| \hat{y}_k = y_s, k\le n \}) \ge l_n m^*_s
	 	\end{equation*}
	 	hold for any $r \in \{1, 2, \cdots, r_0 \}$ and any $s\in \{1, 2, \cdots, s_0^* \}$.
	 	
	 	\begin{proof}[Proof of Claim 1]
	 		Combining \eqref{eq4.5} with \eqref{eq4.10}, for sufficiently large $n \in \mathbb{N}^+$ we have
	 		\begin{equation*}
	 		\frac 1n \sum_{k=1}^n \chi_{\Lambda_r} (T^k x) \ge \mu_x (\Lambda_r) - \beta \ge a n_r - \beta
	 		\end{equation*}
	 		holds for any $r \in \{1, 2, \cdots, r_0 \}$.
	 		
	 		If $n_r = 1$, we have
	 		\begin{equation*}
	 		\sum_{k=1}^n \chi_{\Lambda_r} (T^k x) \ge a n n_r - \beta n = (a - \beta) n.
	 		\end{equation*}
	 		Since $\sum\limits_{k=1}^n \chi_{\Lambda_r} (T^k x)$ is integer and $l_n$ is the minimal integer such that $l_n \ge (a- \beta) n$, we derive that
	 		\begin{equation*}
	 		\#(\{k| \hat{x}_k = x_r, k\le n \}) = \sum_{k=1}^n \chi_{\Lambda_r} (T^k x) \ge l_n = l_n n_r.
	 		\end{equation*}	
	 	
	 		If $n_r >1$, for sufficiently large $n \in \mathbb{N}^+$ we have 
	 		\begin{equation*}
	 		\beta n (n_r -1) - n_r > 0.
	 		\end{equation*}
	 		Then we deduce that 
	 		\begin{equation*}
	 		\begin{split}
	 		\sum_{k=1}^n \chi_{\Lambda_r} (T^k x) 
	 		& \ge a n n_r - \beta n\\
	 		& = l_n n_r -l_n n_r + a n n_r - \beta n n_r + \beta n (n_r -1)\\
	 		& = l_n n_r + \beta n (n_r -1)- \big(l_n - (a - \beta) n \big)n_r\\
	 		& \ge l_n n_r + \beta n (n_r -1)- n_r\\
	 		& > l_n n_r,
	 		\end{split}
	 		\end{equation*}
	 		which is 
	 		\begin{equation*}
	 		\#(\{k| \hat{x}_k = x_r, k\le n \}) = \sum_{k=1}^n \chi_{\Lambda_r} (T^k x) > l_n n_r.
	 		\end{equation*}
	 		
	 		Similarly, we can prove that for sufficiently large $n \in \mathbb{N}^+$, we have
	 		\begin{equation*}
	 	    \#(\{k| \hat{y}_k = y_s, k\le n \}) \ge l_n m^*_s
	 		\end{equation*}
	 		holds for any $s\in \{1, 2, \cdots, s_0^* \}$.
	 	\end{proof}
	 	
	 	By Claim 1, for sufficiently large $n \in \mathbb{N}^+$ there is a set $A \subset \{1, 2, \cdots, n \}$ such that 
	 	\begin{equation}\label{eq4.16}
	 	\# (A) = K l_n
	 	\end{equation}
	 	and
	 	\begin{equation}\label{eq4.17}
	 	\# (\{k| \hat{x}_k = x_r, k \in A \}) = l_n n_r
	 	\end{equation}
	 	holds for any $r \in \{1, 2, \cdots, r_0 \}$.
	 	
	 	Let $\sigma_1 \in S_n$ such that 
	 	\begin{equation*}
	 	\frac 1n \sum\limits_{k=1}^n d(\hat{x}_k, \hat{y}_{\sigma_1 (k)} ) = \inf\limits_{\sigma \in S_n} \frac 1n \sum\limits_{k=1}^n d(\hat{x}_k, \hat{y}_{\sigma(k)} ).
	 	\end{equation*}
	 	For any $s\in \{1,2,\cdots, s_0^*\}$, we denote 
	 	\begin{equation*}
	 	b_s= \#(\{k \le n| \hat{y}_{\sigma_1 (k)} = y_s, k \notin A \} ).
	 	\end{equation*}
	 	Then by \eqref{eq4.16}, we have
	 	\begin{equation} \label{eq4.18}
	 	\sum_{s=1}^{s_0^*} b_s \le n - \#(A) = n-K l_n.
	 	\end{equation}
	 	
     	For any $s\in \{1,2,\cdots, s_0^*\}$ and sufficiently large $n \in \mathbb{N}^+$, we have
	 	\begin{equation*}
	 	\begin{split}
	 	&\hskip0.5cm \frac {\#(\{k \in \mathbb{N^+}| \hat{y}_{\sigma_1 (k)} = y_s, k \in A \}) + b_s}n
	 	= \frac {\#(\{k \in \mathbb{N^+}| \hat{y}_{\sigma_1 (k)} = y_s, k \le n \} )}n\\
	 	&= \frac {\#(\{k \in \mathbb{N^+}| \hat{y}_k = y_s, k \le n \} )}n
	 	=\frac 1n \sum_{k=1}^n\chi_{V_s}(T^k y)\\
	 	&\ge \mu_y (V_s) - \beta  \ \ \big({\text by \ \eqref{eq4.5} \big)  } \\
	 	&\ge m_s a - \beta  \ \ \big({\text by \ \eqref{eq4.11} \big)  } \\
	 	& \ge m_s^* (a - \beta),
	 	\end{split}
	 	\end{equation*}
	 	where the last inequality follows from \eqref{eq4.12}. Therefore, 
	 	\begin{equation}\label{eq4.19}
	 	\begin{split}
	 	\frac {\#(\{k \in \mathbb{N^+}| \hat{y}_{\sigma_1 (k)} = y_s, k \in A \})}n
	 	& \ge m_s^* (a - \beta) - \frac {b_s}n.
	 	\end{split}
	 	\end{equation}
	 	
	 	By Claim 1 and \eqref{eq4.19}, for sufficiently large $n \in \mathbb{N}^+$ there is a $\sigma_2 \in S_n$ such that 
	 	\begin{equation}\label{eq4.20}
	 	\#(\{k| \hat{y}_{\sigma_2 (k)} = y_s, k\in A\}) = l_n m_s^* 
	 	\end{equation}
	 	and
	 	\begin{equation*}
	 	\begin{split}
	 	\frac {\#(\{k \in \mathbb{N^+}| \sigma_2 (k) = \sigma_1 (k),  \hat{y}_{\sigma_1 (k)} = y_s, k \in A \})}n
	 	& \ge m_s^* (a - \beta) - \frac {b_s}n
	 	\end{split}
	 	\end{equation*}
	 	hold for any $s \in \{1,2, \cdots, s_0^*\}$. Since $\# (A) = K l_n$, we deduce that
	 	\begin{equation*}
	 	\begin{split}
	 	\# (\{ k \in \mathbb{N}^+ | \sigma_2 (k) \not= \sigma_1 (k), k \in A \} ) 
	 	& =  \#(A) - \#(\{k \in \mathbb{N^+}| \sigma_2 (k) = \sigma_1 (k), k \in A \}) \\
	 	& \le K l_n - \sum_{s=1}^{s_0^*} \#(\{k \in \mathbb{N^+}| \sigma_2 (k) = \sigma_1 (k),  \hat{y}_{\sigma_1 (k)} = y_s, k \in A \}) \\
	 	& \le K l_n - n \sum_{s=1}^{s_0^*} \big( m_s^* (a - \beta) - \frac {b_s}n \big)\\
	 	& = K l_n - \sum_{s=1}^{s_0^*} m_s^* (a - \beta) n + \sum_{s=1}^{s_0^*} b_s \\
	 	& = K \big(l_n -  (a - \beta) n \big) + \sum_{s=1}^{s_0^*} b_s \ \ (\text{by \eqref{eq4.13} })\\
	 	& \le K + n - K l_n,
	 	\end{split}
	 	\end{equation*}
	 	where the last inequality comes from \eqref{eq4.18} and the fact that $l_n$ is the minimal integer such that $l_n \ge (a - \beta) n$. Then we derive that
	 	\begin{equation*}
	 	\begin{split}
	 	\sum_{k\in A} d(\hat{x}_k, \hat{y}_{\sigma_2 (k)})
	 	& = \sum_{k\in A, \sigma_1 (k) = \sigma_2 (k)} d(\hat{x}_k, \hat{y}_{\sigma_2 (k)}) +
	 	\sum_{k\in A, \sigma_1 (k) \not= \sigma_2 (k)} d(\hat{x}_k, \hat{y}_{\sigma_2 (k)}) \\
	 	& \le \sum_{k\in A} d(\hat{x}_k, \hat{y}_{\sigma_1 (k)}) + M \cdot \# (\{ k \in \mathbb{N}^+ | \sigma_2 (k) \not= \sigma_1 (k), k \in A \} ) \\	
	 	& \le \sum_{k\in A} d(\hat{x}_k, \hat{y}_{\sigma_1 (k)}) + M(K + n - K l_n). 
	 	\end{split}
	 	\end{equation*}
	 	This shows 
	 	\begin{equation*}
	 	\begin{split}
	 	\frac 1n \sum_{k\in A} d(\hat{x}_k, \hat{y}_{\sigma_2 (k)}) 
	 	& \le \frac 1n \sum_{k\in A} d(\hat{x}_k, \hat{y}_{\sigma_1 (k)}) + \frac {MK}n + \frac {M(n- K l_n)}n \\
	 	& \le \frac 1n \sum_{k\in A} d(\hat{x}_k, \hat{y}_{\sigma_1 (k)}) + \frac {MK}n + 2M \varepsilon + MK \beta,
	 	\end{split}
	 	\end{equation*}
	 	where the last inequality comes from \eqref{eq4.15}. 
	 	
	 	Let $B = \{\sigma_2(k) | k \in A \}$ and $H(A,B) = \{ \sigma \in S_n | \sigma (k) \in B, \forall k \in A \}$. Then $\sigma_2 \in H(A,B)$ and we have
	 	\begin{equation}\label{eq4.21}
	 	\begin{split}
	 	\inf_{\sigma \in S_n } \frac 1n \sum_{k =1}^n d(\hat{x}_k, \hat{y}_{\sigma (k)})
	 	&= \frac 1n \sum_{k=1}^n d(\hat{x}_k, \hat{y}_{\sigma_1 (k)}) \ge \frac 1n \sum_{k\in A} d(\hat{x}_k, \hat{y}_{\sigma_1 (k)})\\
	 	& \ge \frac 1n \sum_{k\in A} d(\hat{x}_k, \hat{y}_{\sigma_2 (k)}) - \frac {MK}n - 2M \varepsilon - MK \beta \\
	 	& \ge \inf_{\sigma \in H(A,B) } \frac 1n \sum_{k\in A} d(\hat{x}_k, \hat{y}_{\sigma (k)}) - \frac {MK}n - 2M \varepsilon - MK \beta.
	 	\end{split}
	 	\end{equation} 
	 	
	 	By \eqref{eq4.17}, \eqref{eq4.20} and Lemma \ref{BN_lemma}, we can derive that
	 	\begin{equation}\label{eq4.22}
	 	\inf_{\sigma \in H(A,B) } \sum_{k\in A} d(\hat{x}_k, \hat{y}_{\sigma(k)}) = l_n \cdot \inf_{\sigma \in S_K} \sum_{k=1}^K d(\overline{x}_k, \overline{y}_{\sigma(k)}).
	 	\end{equation}
	 	Combining \eqref{eq4.21} with \eqref{eq4.22}, we estimate the lower bound of $\inf\limits_{\sigma \in S_n } \frac 1n \sum\limits_{k =1}^n d(\hat{x}_k, \hat{y}_{\sigma (k)})$ as follows:
	 	\begin{equation*}
	 	\begin{split}
	 	\inf_{\sigma \in S_n } \frac 1n \sum_{k =1}^n d(\hat{x}_k, \hat{y}_{\sigma (k)})
	 	& \ge \frac {l_n}n \cdot \inf_{\sigma \in S_K} \sum_{k=1}^K d(\overline{x}_k, \overline{y}_{\sigma(k)}) - \frac {MK}n - 2M \varepsilon - MK \beta \\
	 	& \ge ( a - \beta ) \cdot \inf_{\sigma \in S_K} \sum_{k=1}^K d(\overline{x}_k, \overline{y}_{\sigma(k)}) - \frac {MK}n - 2M \varepsilon - MK \beta.
	 	\end{split} 
	 	\end{equation*}
	 	This finishes the proof of Lemma \ref{lemma4.2}.
	 \end{proof}
	
	   \begin{lemma}\label{lemma4.3}
	  	For all sufficiently large $n \in \mathbb{N}^+$, we have
	  	\begin{equation*}
	  	\inf\limits_{\sigma \in S_n} \frac 1n \sum\limits_{k=1}^n d(\hat{x}_k, \hat{y}_{\sigma(k)} ) \le (a- \beta) \cdot \inf\limits_{\sigma \in S_K}\sum\limits_{k=1}^K d(\overline{x}_k, \overline{y}_{\sigma(k)})  + \frac {MK}n +  2M \varepsilon +MK \beta.
	  	\end{equation*}
	  \end{lemma}
      
      \begin{proof}[Proof of Lemma \ref{lemma4.3}]
      	 Let $\sigma_K \in S_K$ such that
      	\begin{equation*}
      	\sum\limits_{k=1}^K d(\overline{x}_k, \overline{y}_{\sigma_ K (k)}) = \inf\limits_{\sigma \in S_K}\sum\limits_{k=1}^K d(\overline{x}_k, \overline{y}_{\sigma(k)}).
      	\end{equation*}
      	For sufficiently large $n \in \mathbb{N}^+$, Claim 1 shows that there is a partition $\{A_i\}_{i=1}^{K+1}$ of $\{1,2, \cdots,n \}$ such that
      	\begin{equation}\label{eq4.23}
      	{\#(A_i)} = l_n \ \ \text{and} \ \
      	\{\hat{x}_k| k\in A_i\} = \{\overline{x}_i\}
      	\end{equation}
      	hold for any $i \in \{1,2, \cdots,K \}$.
      	Similarly, there is a partition $\{B_i\}_{i=1}^{K+1}$ of $\{1,2, \cdots,n \}$ such that
      	\begin{equation}\label{eq4.24}
      	{\#(B_i)} = l_n  \ \ \text{and} \ \
      	\{\hat{y}_k| k\in B_i\} = \{\overline{y}_{\sigma_K(i)}\}
      	\end{equation}
      	hold for any $i \in \{1,2, \cdots,K \}$.
      	
      	By \eqref{eq4.23} and \eqref{eq4.24}, there exists $\sigma_n \in S_n$ such that 
      	\[
        B_i = \{\sigma_n (k) | k \in A_i \}, \ \forall i =1,2, \cdots,K.
      	\] 
      	Thus for any $i \in \{1,2, \cdots,K \}$, we have
      	\begin{equation*}
      	\sum_{k\in A_i} d(\hat{x}_k, \hat{y}_{\sigma_{n}(k)} ) = l_n \cdot d(\overline{x}_i, \overline{y}_{\sigma_K(i)}).
      	\end{equation*}
      	Hence, we deduce that
      	\begin{equation*}
      	\begin{split}
      	\sum_{k=1}^n d(\hat{x}_k, \hat{y}_{\sigma_{n}(k)} ) 
      	&=\sum_{i=1}^{K+1}  \sum_{k\in A_i} d(\hat{x}_k, \hat{y}_{\sigma_{n}(k)} ) \\
      	& = \sum_{i=1}^K  l_n \cdot d(\overline{x}_i, \overline{y}_{\sigma_K(i)}) + \sum_{k \in A_{K+1} }  d(\hat{x}_k, \hat{y}_{\sigma_{n}(k)} )\\
      	& \le l_n \cdot \sum_{i=1}^K d(\overline{x}_i, \overline{y}_{\sigma_K(i)}) + M ( n - K l_n) \\
      	& \le (l_n -1) \cdot \sum_{k=1}^K d(\overline{x}_k, \overline{y}_{\sigma_K(k)}) + MK+ M (n - K l_n).
      	\end{split}
      	\end{equation*}
      	Since $l_n - 1 < (a - \beta)n $, we have
      	\begin{equation*}
        \begin{split}
        \frac 1n \sum_{k=1}^n d(\hat{x}_k, \hat{y}_{\sigma_{n}(k)} ) 
        & \le (a- \beta) \cdot \sum_{k=1}^K d(\overline{x}_k, \overline{y}_{\sigma_K(k)})  + \frac {MK}n + \frac {M(n - K l_n)}n \\
        & \le (a- \beta) \cdot \sum_{k=1}^K d(\overline{x}_k, \overline{y}_{\sigma_K(k)})  + \frac {MK}n + 2 M\varepsilon + MK \beta,
        \end{split}
      	\end{equation*}
      	where the last inequality comes from \eqref{eq4.15}. Hence, we can estimate the upper bound of \\
      	$\inf\limits_{\sigma \in S_n} \frac 1n \sum\limits_{k=1}^n d(\hat{x}_k, \hat{y}_{\sigma(k)} )$ as follows:
      	\begin{equation*}
      	\begin{split}
      	\inf_{\sigma \in S_n} \frac 1n \sum_{k=1}^n d(\hat{x}_k, \hat{y}_{\sigma(k)} ) 
      	&\le \frac 1n \sum_{k=1}^n d(\hat{x}_k, \hat{y}_{\sigma_{n}(k)} ) \\
      	&\le (a- \beta) \cdot \inf_{\sigma \in S_K} \sum_{k=1}^K d(\overline{x}_k, \overline{y}_{\sigma(k)})  + \frac {MK}n + 2M \varepsilon +MK \beta.
      	\end{split}
      	\end{equation*}
      	This finishes the proof of Lemma \ref{lemma4.3}.
      \end{proof}
  
      Given sufficiently large $n \in \mathbb{N}^+$. By \eqref{eq4.3}, Lemma \ref{lemma4.1} and Lemma \ref{lemma4.3}, we deduce that
      \begin{equation}\label{eq4.25}
      \begin{split}
      &\hskip 0.5cm \inf_{\sigma \in S_n} \frac 1n \sum_{k=1}^n d(T^k x,T^{\sigma(k)} y)\\
      &\le (a- \beta) \cdot \inf_{\sigma \in S_K}\sum_{k=1}^K d(\overline{x}_k, \overline{y}_{\sigma(k)})  + \frac {MK}n + 2M \varepsilon +MK \beta + 2 \varepsilon +2M \varepsilon +M(r_0 + s_0) \beta.
      \end{split}
      \end{equation}
      On the other hand, by \eqref{eq4.4}, Lemma \ref{lemma4.1} and Lemma \ref{lemma4.2}, we derive that
      \begin{equation}\label{eq4.26}
      \begin{split}
      &\hskip 0.5cm \inf_{\sigma \in S_n} \frac 1n \sum_{k=1}^n d(T^k x,T^{\sigma(k)} y)\\
      &\ge (a-\beta) \cdot \inf_{\sigma \in S_K} \sum_{k=1}^K d(\overline{x}_k, \overline{y}_{\sigma(k)}) - \frac {MK}n- 2M\varepsilon - MK \beta - 2 \varepsilon -2M \varepsilon - M(r_0 + s_0) \beta.
      \end{split}
      \end{equation}
      
      Let $n \rightarrow +\infty$, and from \eqref{eq4.25} and \eqref{eq4.26} we deduce that
      \begin{equation*}
      \overline{F}(x,y) \le (a- \beta) \cdot \inf_{\sigma \in S_K}\sum_{k=1}^K d(\overline{x}_k, \overline{y}_{\sigma(k)}) + 2M \varepsilon +MK \beta + 2 \varepsilon +2M \varepsilon +M(r_0 + s_0) \beta
      \end{equation*}
	  and
	  \begin{equation*}
	  \underline{F}(x,y) \ge  (a-\beta) \cdot \inf_{\sigma \in S_K} \sum_{k=1}^K d(\overline{x}_k, \overline{y}_{\sigma(k)}) - 2M\varepsilon - MK \beta - 2 \varepsilon -2M \varepsilon -M(r_0 + s_0) \beta.
	  \end{equation*}
	  This shows 
	  \begin{equation*}
	  \overline{F}(x,y) - \underline{F}(x,y) \le 8M \varepsilon +2MK \beta + 4 \varepsilon  +2M(r_0 + s_0) \beta.
	  \end{equation*}
	  
	  Let $\beta \to 0$, then we derive that
	  \begin{equation*}
	  \overline{F}(x,y) - \underline{F}(x,y) \le 8M \varepsilon + 4 \varepsilon= \frac {\alpha}2,
	  \end{equation*}
	  which conflicts with the assumption.  This shows that $F(x,y)$ exists.
	  \end{proof}
	  
	  \section{Invariant measures}\label{5}
	  In this section, we study invariant measures by functions $\overline{F}(x,y)$ and $\underline{F}(x,y)$. And then we prove Theorem \ref{1.3}. Theorem \ref{1.3} is a direct corollary of Theorem \ref{thm5.1} and Theorem \ref{thm5.3}. The following proposition shows that $F(x,y) = 0$ implies the measure sets generated by $x$ and $y$ are the same.

	  \begin{prop}\label{p5.1}
	  	Let $(X,T)$ be a t.d.s. Then $M_x=M_y$ holds for any $x,y \in X$ with $F(x,y)=0$.
	  \end{prop}
	  \begin{proof}
	  	Given $x, y \in X$ with $F(x,y) = 0$. To prove $M_x = M_y$, it suffices to show $M_x \subset M_y$ and $M_y \subset M_x$. In the following, we will prove that $M_x \subset M_y$. With the same reason, $M_y \subset M_x$ holds.
	  	
	  	Given $\mu \in M_x$, there is a subsequence $\{n_r\}_{r=1}^{\infty}$ of positive integers $\mathbb{N^{+}}$ such that for any $f\in C(X)$, we have
	  	\begin{equation}\label{eq5.1}
	  	\lim_{r \to +\infty} \frac 1{n_r}\sum_{k=1}^{n_r}f(T^k x) = \int_X f \mathrm {d}{\mu}.
	  	\end{equation}
	  	
	  	Given $f \in C(X)$ and $\varepsilon > 0$. There is $\delta=\delta(\varepsilon, f) > 0$ such that whenever $x_1, x_2 \in X$ with $d(x_1,x_2) < \delta$, we have
	  	\begin{equation}\label{eq5.2}
	  	|f(x_1)-f(x_2)| < \varepsilon.
	  	\end{equation}
	  	
	  	Let $L = \max\limits _{z \in X}\{|f(z)|\}$. Given $\sigma \in S_{n_r}$, we have 
	  	\begin{equation*}
	  	\begin{split}
	  	& \hskip 0.5cm |\frac1{n_r} \sum_{k=1}^{n_r} f(T^k x) - \frac1{n_r} \sum_{k=1}^{n_r} f(T^k y)|\\ 
	  	& = \frac 1{n_r} |\sum_{k=1}^{n_r} \big(f(T^k x) - f(T^{\sigma(k)} y)\big)|
	  	\leq \frac 1{n_r} \sum_{k=1}^{n_r} |f(T^k x) - f(T^{\sigma(k)} y)|\\
	  	&\leq \varepsilon \times \frac {\# \big(\{k \in \mathbb{N^{+}}|d(T^k x,T^{\sigma(k)} y) < \delta, k \le n_r \}\big)}{n_r} \ \  \big({\text by \ \eqref{eq5.2}} \big) \\
	  	& \hskip 0.5cm +2L \times \frac {\# \big(\{k \in \mathbb{N^{+}}|d(T^k x,T^{\sigma(k)} y) \geq \delta, k \le n_r \}\big)}{n_r} \\
	  	& \le \varepsilon +2L \times \frac {\# \big(\{k \in \mathbb{N^{+}}|d(T^k x,T^{\sigma(k)} y) \geq \delta, k \le n_r \}\big)}{n_r}.
	  	\end{split}
	  	\end{equation*}
	  	Since
	  	\[
	  	\frac 1{n_r} \sum_{k=1}^{n_r} d(T^k x,T^{\sigma(k)} y) \geq \delta \times \frac {\# \big( \{k \in \mathbb{N^{+}}|d(T^k x,T^{\sigma(k)} y) \geq \delta, k \le n_r \} \big)}{n_r},
	  	\]
	  	we deduce that
	  	\[
	  	|\frac1{n_r} \sum_{k=1}^{n_r} f(T^k x) - \frac1{n_r} \sum_{k=1}^{n_r} f(T^k y)| \leq \varepsilon + \frac {2L}{\delta} \times \frac 1{n_r} \sum_{k=1}^{n_r} d(T^k x,T^{\sigma(k)} y).
	  	\]
	  	Thus we have
	  	\[
	  	|\frac1{n_r} \sum_{k=1}^{n_r} f(T^k x) - \frac1{n_r} \sum_{k=1}^{n_r} f(T^k y)| \leq \varepsilon + \frac {2L}{\delta} \times \inf_{\sigma \in S_{n_r}} \frac 1{n_r} \sum_{k=1}^{n_r} d(T^k x,T^{\sigma(k)} y).
	  	\]
	  	Let $r \to +\infty$, then we have
	  	\begin{equation*}
	  	\begin{split}
	  	\limsup_{r \to +\infty} |\frac1{n_r} \sum_{k=1}^{n_r} f(T^k x) - \frac1{n_r} \sum_{k=1}^{n_r} f(T^k y)| 
	  	& \leq \varepsilon + \frac {2L}{\delta} \times \limsup_{r \to +\infty} \inf_{\sigma \in S_{n_r}} \frac 1{n_r} \sum_{k=1}^{n_r} d(T^k x,T^{\sigma(k)} y) \\
	  	& = \varepsilon + \frac {2L}{\delta} \times F(x,y)  = \varepsilon.
	  	\end{split}
	  	\end{equation*}
	  	Let $\varepsilon \to 0$, then we deduce that
	  	\[
	  	\limsup_{r \to +\infty} |\frac1{n_r} \sum_{k=1}^{n_r} f(T^k x) - \frac1{n_r} \sum_{k=1}^{n_r} f(T^k y)| = 0.
	  	\]
	  	Combining this with \eqref{eq5.1}, we have
	  	\[
	  	\lim_{r \to +\infty} \frac 1{n_r}\sum_{k=1}^{n_r}f(T^k y) =\lim_{r \to +\infty} \frac 1{n_r}\sum_{k=1}^{n_r}f(T^k x) = \int_X f \mathrm {d}{\mu},
	  	\]
	  	which implies $\mu \in M_y$. Therefore, $M_x \subset M_y$. This finishes the proof of Proposition \ref{p5.1}.
	  \end{proof}
	  
	      With respect to $\underline{F}(x,y)$, we have the following result similar to Proposition \ref{p5.1}.
	  \begin{prop}\label{p5.2}
	  	Let $(X,T)$ be a t.d.s. Then $M_x \cap M_y \not= \O$ for any $x, y \in X$ with $\underline{F}(x,y)=0$.
	  \end{prop}
	  \begin{proof}
	  	Given $x,y \in X$ with $\underline{F}(x,y)=0$. There is a subsequence $\{n_r\}_{r=1}^{\infty}$ of positive integers $\mathbb{N^{+}}$ such that
	  	\[
	  	\lim_{r \to +\infty} \inf_{\sigma \in S_{n_r}} \frac 1{n_r} \sum_{k=1}^{n_r} d(T^k x,T^{\sigma(k)} y) = 0.
	  	\]
	  	Without loss of generality, we can assume that there exists $\mu \in M_x$ such that 
	  	\[
	  	\lim_{r \to +\infty} \frac 1{n_r}\sum_{k=1}^{n_r}f(T^k x) = \int_X f \mathrm {d}{\mu}
	  	\]
	  	holds for any $f \in C(X)$. 
	  	
	  	With the same reason in the proof of Proposition \ref{p5.1}, we derive that $\mu \in M_y$. Thus, $M_x \bigcap M_y \not= \O$.
	  \end{proof}
	
		When $x \in X$ is a generic point of $(X,T)$, we can strengthen the Proposition~\ref{p5.1} as follows.
		\begin{prop}\label{p5.3}
			Let $(X,T)$ be a  t.d.s. and $x,y \in X$. If $x$ is a generic point of $(X,T)$, then $F(x,y)=0$ if and only if $M_x=M_y$.
		\end{prop}
		\begin{proof}
			Proposition \ref{p5.1} shows that $F(x,y)=0$ implies $M_x=M_y$. To finish the proof of Proposition \ref{p5.3}, we need only to prove that $M_x=M_y$ implies $F(x,y)=0$.
			
			Let $\mu= \lim\limits_{n \to +\infty} \frac 1n \sum\limits_{k=1}^n \delta_{T^k x}$, then $M_x= M_y= \{\mu\}$. Given $\varepsilon > 0$, and let $\eta = \frac {\varepsilon}{1+M}$, where $M=diam(X)$. By Lemma~\ref{measure lemma}, there are mutually disjoint open sets $\{U_s \}_{s=1}^{s_0}$
			such that
			\begin{equation*}
			\mu(\bigcup_{s=1}^{s_0} U_s) \geq 1- \eta \ \ \text{and} \ \ diam(U_s) \leq \eta, {\forall}s=1,2,\cdots,s_0.
			\end{equation*}
		
			Combining Lemma \ref{lemma2.1} with Proposition \ref{P3.4}, we have
			\begin{equation*}
			\overline{F} (x,y) \le \eta \sum\limits_{s=1}^{s_0} \mu(U_s) + M \big(1- \sum\limits_{s=1}^{s_0} \mu(U_s) \big) \le \eta + M \eta =\varepsilon .
			\end{equation*}
			Let $\varepsilon \to 0$, then we deduce that $\overline{F} (x,y) = 0$. This shows $F(x,y) = 0$.
		\end{proof}
	
	    Applying Proposition~\ref{p5.3}, we have the following theorem.
	    
	    \begin{theorem}\label{thm5.1}
	    	Let $(X,T)$ be a t.d.s. Then $(X,T)$ is uniquely ergodic if and only if $N(F) = X \times X$.
	    \end{theorem}
	    \begin{proof}
	    	If $(X,T)$ is uniquely ergodic and $\mu$ is the unique ergodic measure. Then for any $x,y \in X$, we have $M_x=M_y=\{\mu \}$, which implies that $x$ and $y$ are generic points. By Proposition \ref{p5.3}, we derive that $F(x,y)=0$. Thus $(x,y) \in N(F)$. Hence $N(F)=X \times X$.
	    	
	    	If $N(F)= X \times X$. Let $\mu_1$ and $\mu_2$ be ergodic measures on $(X,T)$. By Birkhoff pointwise ergodic theorem, there exist $x,y \in X$ such that $M_x = \{\mu_1\}$ and $M_y= \{\mu_2\}$. Since $N(F)=X \times X$, we have $F(x,y)=0$. By Proposition \ref{p5.3}, we deduce that $M_x = M_y$, which implies $\mu_1 = \mu_2$. Thus, $(X,T)$ is uniquely ergodic. 
	    \end{proof}
    
        When $(X,T)$ is a transitive weak mean equicontinuous system, we can deduce that $N(F) = X\times X$. Thus by Theorem \ref{thm5.1}, we have that a transitive weak mean equicontinuous system is uniquely ergodic.
        
        \begin{corollary}
        	Let $(X,T)$ be a transitive weak mean equicontinuous t.d.s. Then $(X,T)$ is uniquely ergodic. In particular, a transitive mean equicontinuous system is uniquely ergodic.
        \end{corollary}
        \begin{proof}
        	Let $x \in X$ be a transitive point of $(X,T)$. Then for any $y \in X$, there is a subsequence $\{m_r \}_{r=1}^{+\infty}$ of positive integers $\mathbb{N^{+}}$ such that $\lim\limits_{r \to +\infty} T^{m_r} x = y$. Since $(X,T)$ is weak mean equicontinuous, we deduce that $\lim\limits_{r \to +\infty} \overline{F}(T^{m_r} x, y) = 0$. By Proposition \ref{P3.2}, we have that for any $r \ge 1$, $\overline{F}(x,y) = \overline{F}(T^{m_r} x,y)$. Thus, $\overline{F}(x,y) = 0$. 
        	
        	 Given $y_1, y_2 \in X$. By Proposition \ref{P3.1}, we have
        	 \begin{equation*}
        	 \overline{F}(y_1,y_2) \le \overline{F}(y_1, x) + \overline{F}(x, y_2) = 0,
        	 \end{equation*}
        	 which shows $(y_1,y_2) \in N(F)$. Thus $N(F)=X\times X$. By Theorem \ref{thm5.1}, we derive that $(X, T)$ is uniquely ergodic.
        \end{proof}
    
        Combining Theorem \ref{thm5.1} with Proposition~\ref{p5.2}, we can show a new characterization of unique ergodicity by $N(\underline{F})$. 
        \begin{theorem}\label{thm5.3}
        	Let $(X,T)$ be a t.d.s. Then $(X,T)$ is uniquely ergodic if and only if $N(\underline{F})=X \times X$.
        \end{theorem}
        \begin{proof}
        	Assume that $(X,T)$ is uniquely ergodic. By Theorem \ref{thm5.1}, we have $N(F) = X \times X$. Since $N(F) \subset N(\underline{F})$, we derive that $N(\underline{F}) = X \times X$.
        	
        	Conversely, we assume that $N(\underline{F}) = X \times X$. Let $\mu_1$ and $\mu_2$ be ergodic measures of $(X,T)$. By Birkhoff pointwise ergodic theorem, there are $x,y \in X$ such that $M_x = \{\mu_1\}$ and $M_y= \{\mu_2\}$. Since $N(\underline{F})= X \times X$, we have $\underline{F}(x,y)=0$. By Proposition \ref{p5.2}, we deduce that $M_x \cap M_y \not=\O$, which implies $\mu_1=\mu_2$. Thus, $(X,T)$ is uniquely ergodic. 
        \end{proof}
      
        Proposition \ref{p5.1} shows that $N(F)$ is a subset of all point pairs in $X$ which can generate the same measure set. Combining Proposition \ref{p5.3} with the fact that $Q$ is a set of invariant measure one, it is reasonable to regard $N(F)$ as the whole set of all point pairs in $X$ which can generate the same measure set in the view of measure theory. Thus, there are close connections between $N(F)$ and invariant measures. Similarly, there are close connections between $N(\underline{F})$ and invariant measures. We state some of them as follows.
        \begin{theorem}\label{thm5.4}
        	Let $(X,T)$ be a t.d.s. and $\mu \in M(X,T)$. Then the following statements are equivalent:
        	\begin{itemize}
        		\item [(\emph{1})]
        		$\mu$ is ergodic;
        		\item [(\emph{2})]
        		$(\mu \times \mu) \big(N(F) \big)= 1$;
        		\item [(\emph{3})]
        		$(\mu \times \mu) \big(N(\underline{F}) \big)= 1$.
        	\end{itemize}
        \end{theorem}
        \begin{proof}
        	(3) $\Rightarrow$ (2) Since $Q$ has invariant measure one, we have $(\mu \times \mu) (Q \times Q) =1$. Thus, $(\mu \times \mu) \big(N(\underline{F}) \bigcap (Q \times Q) \big) =1$. By Theorem~\ref{1.2}, we have $N(F)\bigcap (Q \times Q) = N(\underline{F}) \bigcap  (Q \times Q)$. Hence, $(\mu \times \mu) \big(N(F) \bigcap (Q \times Q) \big) =1$. This shows $(\mu \times \mu) \big(N(F) \big)= 1$.

        	(2) $\Rightarrow$ (1) Since
        	\[
        	(\mu \times \mu) \big(N(F) \big) =\int_X \mu \big(N(F,x) \big) \mathrm {d}{\mu(x)},
        	\]
        	there is $x_0 \in X$ such that $\mu \big(N(F,x_0) \big)=1$. 
        	For $Q_T$ is a Borel set of invariant measure one, there is $y \in Q_T \bigcap N(F,x_0)$. Let $\mu_y= \lim\limits_{n \to +\infty} \frac 1n \sum\limits_{k=1}^n \delta_{T^k y}$. Then by Proposition~\ref{p5.3}, we can derive that $M_z = M_y =\{\mu_y \}$ for any $z \in N(F,x_0)$. Thus given $f \in C(X)$, we have  
        	\begin{equation*}
        	\lim\limits_{n \to +\infty} \frac 1n \sum\limits_{k=1}^n f(T^k z) = \int_{X} f \mathrm {d}{\mu_y}
        	\end{equation*}
        	holds for any $z \in N(F,x_0)$.
            Hence, we deduce that
        	\begin{equation*}
        	\begin{split}	
        	\int_{X} f \mathrm {d}{\mu}& = \lim_{n \to +\infty} \int_{N(F,x_0)} \frac 1n \sum_{k=1}^n f(T^k z) \mathrm {d}{\mu(z)}\\
        	&=  \int_{N(F,x_0)} \lim_{n \to +\infty} \frac 1n \sum_{k=1}^n f(T^k z) \mathrm {d}{\mu(z)}\\
        	&=\int_{X} f \mathrm {d}{\mu_y}.
        	\end{split}
        	\end{equation*}
        	This shows $\mu =\mu_y$. So $\mu$ is an ergodic measure.	
        	
        	(1) $\Rightarrow$ (3) By Birkhoff pointwise ergodic  theorem and Proposition~\ref{p5.3}, there exists a measurable subset $\Lambda$ of $X$ such that  $\mu(\Lambda)=1$ and $\Lambda \times \Lambda \subset N(F)$. Thus $(\mu \times \mu) \big(N(F) \big) = 1$. Since $N(F) \subset N(\underline{F})$, we derive that $(\mu \times \mu) \big( N(\underline{F}) \big) = 1$.
        \end{proof}
    
        \begin{theorem}\label{thm5.5}
        	Let $(X,T)$ be a t.d.s. and $m \in M(X)$. Then the following statements are equivalent:
        	\begin{itemize}
        		\item [(\emph{1})]
        		$(X,T)$ has physical measures with respect to $m$;
        		\item [(\emph{2})]
        		$(m\times m) \big(N(F)\bigcap (Q \times Q) \big)>0$;
        		\item [(\emph{3})]
        		$(m\times m) \big(N(\underline{F}) \bigcap (Q \times Q) \big)>0$.
        	\end{itemize}
        \end{theorem} 
        
        \begin{proof}
        	(2) $\Leftrightarrow$ (3) By Theorem~\ref{1.2}, we have 
        	\begin{equation*}
        	N(F)\bigcap (Q \times Q) = N(\underline{F}) \bigcap  (Q \times Q).
        	\end{equation*}
        	
        	(1) $\Rightarrow$ (2) Let $\mu$ be a physical measure of $(X,T)$ with respect to $m$. Then $m \big(B(\mu) \big)>0$. For any $x,y \in B(\mu)$, we have $M_x = M_y = \{\mu \}$. Thus $x,y \in Q$. By Proposition~\ref{p5.3}, we have $F(x,y)=0$, which shows $(x,y) \in N(F) \bigcap (Q\times Q)$. Thus, 
        	\begin{equation*}
        	B(\mu) \times B(\mu) \subset N(F) \bigcap (Q\times Q). 
        	\end{equation*}
        	Hence we have
        	\[
        	(m \times m) \big(N(F)\bigcap (Q\times Q) \big) \ge (m \times m) \big(B(\mu)\times B(\mu) \big) >0.
        	\]
        	
        	(2) $\Rightarrow$ (1)  There is $x_0 \in Q$ such that $m \big(N(F,x_0) \big)>0$. If not, for any $x \in Q$, we have $m \big(N(F,x) \big)=0$. Then we derive that
        	\[
        	(m \times m) \big(N(F)\bigcap (Q\times Q) \big) = \int_Q m \big(N(F,x) \big) \mathrm {d}{m(x)} =0,
        	\]
        	which is a contradiction.
        	
        	Let $\mu_{x_0}$ be the invariant measure generated by $x_0$. Then by Proposition~\ref{p5.3}, we have $B(\mu_{x_0})=N(F,x_0)$. Thus 
        	\begin{equation*}
        	m \big(B(\mu_{x_0}) \big)=m \big(N(F,x_0) \big) >0,
        	\end{equation*}
        	which shows that $\mu_{x_0}$ is a physical measure with respect to $m$.
        \end{proof}
        
        \section{Weak mean equicontinuity} \label{6}
        In this section, we study $\overline{F}$-continuity and $F$-continuity. Combining the following Proposition \ref{p6.1} with Theorem \ref{1.2}, we deduce that $\overline{F}$-continuity is equivalent to $F$-continuity. Then we provide the proof of Theorem \ref{1.5}.
        
        \begin{prop}\label{p6.1}
        	Let $(X,T)$ be an $\overline{F}$-continuous t.d.s. Then all the points in $X$ are generic points.
        \end{prop}
        \begin{proof}
        	Given $x \in X$. For any $y \in \overline{N(F,x)}$, there are $\{y_n\}_{n=1}^{\infty} \subset N(F,x)$ such that $\lim\limits_{n \to \infty} y_n = y$.
        	By Proposition \ref{P3.1}, we have 
        	\[
        	\overline{F}(x,y) \leq \overline{F}(x,y_n) + \overline{F}(y_n, y) = \overline{F}(y_n, y).
        	\]
        	Since $(X,T)$ is $\overline{F}$-continuous, we have $\lim\limits_{n \to +\infty} \overline{F}(y_n,y) = 0$. Thus we deduce that
        	\begin{equation*}
        	F(x,y)=\overline{F}(x,y) = 0, 
        	\end{equation*}
        	which implies $y \in N(F,x)$. Hence $N(F,x)$ is closed. By Proposition \ref{P3.3}, we know $N(F,x)$ is an invariant set. Then by Theorem \ref{1.3}, we derive that $ \big(N(F,x),T \big)$ is uniquely ergodic, which implies that all the points in $N(F,x)$ are generic. In particular, $x$ is a generic point.
        \end{proof}
        Since an $F$-continuous t.d.s is $\overline{F}$-continuous, the following is a direct corollary of Proposition \ref{p6.1}
        \begin{prop}
        	Let $(X,T)$ be an $F$-continuous t.d.s. Then all the points in $X$ are generic points.
        \end{prop}
        
        By Theorem \ref{1.2} and Proposition~\ref{p6.1}, we can deduce the following theorem.
        \begin{theorem}
        	Let $(X,T)$ be a t.d.s. Then $(X,T)$ is $\overline{F}$-continuous if and only if $(X,T)$ is $F$-continuous.
        \end{theorem}
        \begin{proof}
        	We need only to prove that $\overline{F}$-continuity implies $F$-continuity. Suppose that $(X,T)$ is $\overline{F}$-continuous. By Proposition~\ref{p6.1}, we derive that all the points in $X$ are generic points. Then by Theorem \ref{1.2}, we have $F(x,y)$ exists for any $x,y \in X$. Thus $(X,T)$ is $F$-continuous.
        \end{proof}
        
        Next, we give the proof of Theorem~\ref{1.5}.
        \begin{proof}[\bf Proof of Theorem~\ref{1.5}]
        	Assume that $(X,T)$ is weak mean equicontinuous, then we will prove $f^*$ is continuous for any $f \in C(X)$.
        	
        	Given $f \in C(X)$, by Proposition~\ref{p6.1} we know $f^* (x)$ exists for any $x \in X$. Fix $\upsilon > 0$, then there is $\varepsilon > 0$ such that whenever $x,y \in X$ with $d(x,y) < \varepsilon$, we have
        	\begin{equation}\label{eq6.1}
        	|f(x) - f(y)| < \frac{\upsilon}{2(2L+1)},
        	\end{equation}
        	where $L=\max\limits_{x \in X} \{|f(x)|\}$.
        	
        	Given $x,y \in X$. For any $\sigma \in S_n$, we have
        	\begin{equation*}
        	\begin{split}
        	&\hskip 0.5cm |\frac 1n \sum_{k=1}^n f(T^k x) - \frac 1n \sum_{k=1}^n f(T^k y)| \\
        	&=\frac 1n |\sum\limits_{k=1}^n \big(f(T^k x) - f(T^{\sigma(k)} y) \big)|  \leq \frac 1n \sum_{k=1}^n |f(T^k x - f(T^{\sigma(k)} y)|\\
        	&\le \frac{\upsilon}{2(2L+1)} \times \frac {\# \big(\{k \in \mathbb{N^{+}}|d(T^k x,T^{\sigma(k)} y) < \varepsilon, k\le n\} \big)}n \ \ \big({\text by \ \eqref{eq6.1}} \big) \\
        	& \hskip 0.5cm + 2L \times \frac {\# \big(\{k \in \mathbb{N^{+}}|d(T^k x,T^{\sigma(k)} y) \geq \varepsilon, k\le n\} \big)}n \\
        	&\le \frac{\upsilon}{2(2L+1)} + 2L \times \frac {\# \big(\{k \in \mathbb{N^{+}}|d(T^k x,T^{\sigma(k)} y) \geq \varepsilon, k\le n\} \big)}n.
        	\end{split}
        	\end{equation*}
        	Since 
        	\[
        	\frac 1n \sum_{k=1}^n d(T^k x,T^{\sigma(k)} y) \geq \varepsilon \times \frac {\# \big(\{k \in \mathbb{N^{+}}|d(T^k x,T^{\sigma(k)} y) \geq \varepsilon, k\le n\} \big)}n,
        	\]
        	we deduce that
        	\begin{equation*}
        	|\frac 1n \sum_{k=1}^n f(T^k x) - \frac 1n \sum_{k=1}^n f(T^k y)| 
        	\le \frac{\upsilon}{2(2L+1)} + \frac {2L}{\varepsilon} \times \frac 1n \sum_{k=1}^n d(T^k x,T^{\sigma(k)} y).
        	\end{equation*}
        	Thus we have
        	\begin{equation*}
        	|\frac 1n \sum_{k=1}^n f(T^k x) - \frac 1n \sum_{k=1}^n f(T^k y)| 
        	\le \frac{\upsilon}{2(2L+1)} + \frac {2L}{\varepsilon} \times \inf_{\sigma \in S_n} \frac 1n \sum_{k=1}^n d(T^k x,T^{\sigma(k)} y).
        	\end{equation*}
        	Let $n \to +\infty$, then we have 
        	\begin{equation*}
        	\limsup_{n \to +\infty} |\frac 1n \sum_{k=1}^n f(T^k x) - \frac 1n \sum_{k=1}^n f(T^k y)| 
        	\le \frac{\upsilon}{2(2L+1)} + \frac {2L}{\varepsilon} \times \overline{F}(x,y).
        	\end{equation*}
        	This implies
        	\begin{equation}\label{eq6.2}
        	|f^* (x) - f^* (y) | \le \frac{\upsilon}{2(2L+1)} + \frac {2L}{\varepsilon} \times \overline{F}(x,y).
        	\end{equation}
        	
        	Since $(X,T)$ is weak mean equicontinuous, there is $\delta > 0$ such that whenever $x,y \in X$ with $d(x,y) < \delta$, we have
        	\begin{equation*}
        	\overline{F}(x,y) < \frac{\varepsilon \upsilon}{2(2L+1)}.
        	\end{equation*}
        	Combining this with \eqref{eq6.2}, we deduce that
        	\begin{equation*}
        	|f^* (x) - f^* (y) |
        	\le \frac{\upsilon}{2(2L+1)} + \frac {2L}{\varepsilon} \times \frac{\varepsilon \upsilon}{2(2L+1)} = \frac {\upsilon}2 < \upsilon
        	\end{equation*}
        	whenever $x, y \in X$ with $d(x,y) < \delta$,
        	which implies $f^{*}(x) \in C(X)$.
        	\medskip
        	
        	Conversely, we will prove $(X,T)$ is weak mean equicontinuous with the assumption that $f^*$ is continuous for any $f \in C(X)$. 
        	
        	If $(X,T)$ is not weak mean equicontinuous, there are  $x \in X$, $\varepsilon > 0$ and $\{x_m\}_{m=1}^{\infty} \subset X$ such that $\lim\limits_{m \to +\infty}x_m= x$ but $\overline{F}(x,x_m) \geq \varepsilon$. With the assumption, we know $x$ and $\{x_m \}_{m=1}^{\infty}$ are generic points.

        	Let 
        	$\mu = \lim\limits_{n \to + \infty} \frac 1n \sum\limits_{k=1}^n \delta_{T^k x}$, $M=diam(X), \eta = \frac{\varepsilon}{2(M+5)}$. Then by Lemma~\ref{measure lemma}, there exist finite mutually disjoint closed subsets $\{\Lambda_s\}_{s=1}^{s_0}$ of $X$ such that
        	\[
        	\mu(\bigcup_{s=1}^{s_0}\Lambda_s) \textgreater 1 - \eta \ \ \text{and} \ \ diam(\Lambda_s) \textless \eta,  \ {\forall} s = 1,2,\cdots,s_0.
        	\]
        	Take $\delta = \min\limits_{s_1 \not= s_2}\{d(\Lambda_{s_1},\Lambda_{s_2})\}$, $r=\min\{\frac{\delta}5,\eta\}$ and $\alpha =\frac{\varepsilon}{4Ms_0 + 1}$. For any $s \in \{1,2,\cdots,s_0 \}$, let $U_s = \{y \in X : d(y, \Lambda_s) <r \}$ and $V_s = \{y \in X: d(y, \Lambda_s) < 2r \}$. Then $\{U_s \}_{s=1}^{s_0}$ and $\{V_s \}_{s=1}^{s_0}$ are  mutually disjoint open subsets of $X$ and $diam(V_s) \le 5 \eta$ for any $s \in \{1,2,\cdots,s_0 \}$. We have the following claim:
        	
        	\medskip
        	\noindent {\bf Claim 2.}
        	For any $x_m \in X$, there is $s_m \in \{1,2,\cdots,s_0\}$ such that
        	\[
        	\liminf_{n \to +\infty} \frac 1n \sum_{k=1}^n \chi_{U_{s_m}}(T^k x) > \liminf_{n \to +\infty} \frac 1n \sum_{k=1}^n \chi_{V_{s_m}}(T^k x_m)  + \alpha.
        	\]
        	
        	\begin{proof}[Proof of Cliam 2]
        		If not, there is $x_m \in X$ such that for any $s \in \{1,2,\cdots,s_0 \}$, we have
        		\begin{equation} \label{eq6.3}
        		\liminf_{n \to +\infty} \frac 1n \sum_{k=1}^n \chi_{U_{s}}(T^k x) \le  \liminf_{n \to +\infty} \frac 1n \sum_{k=1}^n \chi_{V_{s}}(T^k x_m)  + \alpha.
        		\end{equation}
        		Given $s \in \{1,2,\cdots,s_0 \}$. By Lemma \ref{lemma2.1}, we have 
        		\begin{equation}\label{eq6.4}
        		\liminf_{n \to +\infty} \frac 1n \sum_{k=1}^n \chi_{U_{s}}(T^k x) \ge \mu(U_s).
        		\end{equation}
        		Then combining \eqref{eq6.3} with \eqref{eq6.4}, we deduce that
        		\begin{equation*}
        		\begin{split}
        		\liminf_{n \to +\infty} \frac 1n \sum_{k=1}^n \chi_{V_{s}}(T^k x_m) 
        		&\ge \liminf_{n \to +\infty} \frac 1n \sum_{k=1}^n \chi_{U_{s}}(T^k x) - \alpha \\
        		& \ge \mu(U_s) - \alpha. 
        		\end{split}
        		\end{equation*}
        		Since $U_s \subset V_s$, we have
        		\begin{equation*}
        		\liminf_{n \to +\infty} \frac 1n \sum_{k=1}^n \chi_{V_{s}}(T^k x) \ge \liminf_{n \to +\infty} \frac 1n \sum_{k=1}^n \chi_{U_{s}}(T^k x) \ge \mu(U_s).
        		\end{equation*}
        		Then by Proposition \ref{P3.4}, we derive that
        		\begin{equation*}
        		\begin{split}
        		\overline{F}(x, x_m) 
        		&\le 5 \eta \sum\limits_{s=1}^{s_0} \big(\mu(U_s) - \alpha \big) + M \Big(1- \sum\limits_{s=1}^{s_0} \big(\mu(U_s) - \alpha \big) \Big)\\
        		& \le 5 \eta + M \eta + M s_0 \alpha \le \frac {3 \varepsilon}4,
        		\end{split}
        		\end{equation*}
        		which is a contradiction. This finishes the proof of Claim 2.
        	\end{proof}
        	
        	By Claim 2, there is $s_{m_0} \in \{1,2,\cdots,s_0\}$ and a subsequence $\{x_{m_p}\}_{p=1}^{\infty}$ of $\{x_m\}_{m=1}^{\infty}$ such that for any $p \in \mathbb{N}^+$, we have
        	\begin{equation}\label{eq6.5}
        	\liminf_{n \to +\infty} \frac 1n \sum_{k=1}^n \chi_{U_{s_{m_0}}}(T^k x) > \liminf_{n \to +\infty} \frac 1n \sum_{k=1}^n \chi_{V_{s_{m_0}}}(T^k x_{m_p})  + \alpha.
        	\end{equation}
        	Take $f \in C(X)$ such that $0\le f \le 1$ and
        	\[
        	f|_{\overline{U}_{s_{m_0}}} = 1, \ \ f|_{V_{s_{m_0}}^c} = 0.
        	\]
        	Then we derive that
        	\[
        	\liminf_{n \to +\infty} \frac 1n \sum_{k=1}^n \chi_{U_{s_{m_0}}}(T^k x)  \le \liminf_{n \to +\infty} \frac 1n \sum_{k=1}^n f(T^k x) = f^{*}(x)
        	\]
        	and
        	\[
        	\liminf_{n \to +\infty} \frac 1n \sum_{k=1}^n \chi_{V_{s_{m_0}}}(T^k x_{m_p})  \ge \liminf_{n \to +\infty} \frac 1n \sum_{k=1}^n f(T^k x_{m_p}) = f^{*}(x_{m_p}).
        	\]
        	Thus by \eqref{eq6.5}, we deduce that
        	\[
        	f^{*}(x) \ge  f^{*}(x_{m_p}) + \alpha,
        	\]
        	which implies $f^{*}(x) \notin C(X)$. This is a contradiction. Hence $(X,T)$ is weak mean equicontinuous.	
        \end{proof}


\begin{thebibliography}{99}
		\bibitem{EJK}
		\newblock E.Akin, J. Auslander and K.Berg,
		\newblock When is a transitive map chaotic?,
		\newblock \emph{Convergence in Ergodic Theory and Probility,de Gruyter, Berlin}, \textbf{5} (1996), 25--40.
		
		\bibitem{PJ}
		\newblock P. Halmos and J. Von Neumann,
		\newblock Operator methods in classical mechanics, \uppercase\expandafter{\romannumeral2},
		\newblock \emph{Ann. of Math. (2)}, \textbf{43} (1942), 332--350.
		
		\bibitem{S}
		\newblock S. Fomin,  
		\newblock On dynamical systems with a purely point spectrum,
		\newblock \emph{Doklady Akad. Nauk SSSR}, \textbf{77} (1951), 29--32 (In Russian).
		
		\bibitem{JA}
		\newblock J. Auslander,  
		\newblock Mean-L-stable systems,
		\newblock \emph{Illinois J. Math.}, \textbf{3} (1959), 566--579.
		
		\bibitem{JCO}
		\newblock J. C. Oxtoby,  
		\newblock Ergodic sets,
		\newblock \emph{Bull. Amer. Math. Soc.}, \textbf{58} (1952), 116--136.
		
		\bibitem{FG1}
		\newblock F. Garcia-Ramos,  
		\newblock A characterization of $\mu$-equicontinuity for topological dynamical systems
		\newblock \emph{Proc. Amer. Math. Soc.}, \textbf{145} (2017), 3357--3368.
		
		\bibitem{FG2}
		\newblock F. Garcia-Ramos,  
		\newblock Weak forms of topological and measure-theoretical equicontinuty: relationships with discrete specturm and sequence entropy,
		\newblock \emph{Ergodic Theory Dynam. Systems}, \textbf{37} (2017), no. 4 1211--1237.
		
		\bibitem{FGL}
		\newblock F. Garcia-Ramos and L. Jin,  
		\newblock Mean proximality and mean Li-Yorke chaos,
		\newblock \emph{Proc. Amer. Math. Soc}, \textbf{145} (2017), no. 7 2959--2969.
		
		\bibitem{WJLX}
		\newblock W. Huang, J. Li, J. Thouvenot, L. Xu and X. Ye,  
		\newblock Bounded complexity, mean equicontinuity and discrete spectrum, 
		\newblock Arxiv: 1806.02980.
		
		\bibitem{JL}
		\newblock J. Li, 
		\newblock How chaotic is an almost mean equicontinuous system? 
		\newblock \emph{Discrete Contin. Dyn. Syst.}, \textbf{38} (2018), no. 9 4727--4744.
		
		\bibitem{JSX}
		\newblock J. Li, S. Tu and X. Ye,  
		\newblock Mean equicontinuity and mean sensitivity,
		\newblock \emph{Ergodic Theory Dynam. Systems}, \textbf{35} (2015), no. 9 2587--2612.	
		
		\bibitem{PW} 
		\newblock P. Walters,
		\newblock \emph{An Introduction to Ergodic Theory},
		\newblock  Graduate Texts in Mathematics, \textbf{79}, Springer-Verlag, New York-Berlin, 1982.
		
		\bibitem{AB} 
		\newblock A. Barvinok,
		\newblock \emph{A Course in Convexity},
		\newblock  Graduate Studies in Mathematics, \textbf{54}, American Mathematical Society, Rhode Island-Providence, 2002.
		
	\end{thebibliography}
\end{document}